\documentclass[11pt,reqno,letterpaper]{amsart}

\usepackage{amssymb,amsmath,amsthm,amsfonts}
\usepackage{bbm,enumerate}

\usepackage{bookmark}
\usepackage{hyperref}
\hypersetup{pdfstartview={FitH}}

\theoremstyle{plain}
\newtheorem{theorem}{Theorem}
\newtheorem{problem}[theorem]{Problem}
\newtheorem{lemma}[theorem]{Lemma}

\theoremstyle{remark}
\newtheorem{remark}{Remark}

\numberwithin{equation}{section}

\renewcommand{\leq}{\leqslant}
\renewcommand{\geq}{\geqslant}

\begin{document}

\title[$\textup{L}^p$ estimates for unimodular multipliers]{Asymptotic behavior of $\textup{L}^p$ estimates for a class of multipliers with homogeneous unimodular symbols}

\author[A. Bulj]{Aleksandar Bulj}
\author[V. Kova\v{c}]{Vjekoslav Kova\v{c}}

\address{Department of Mathematics, Faculty of Science, University of Zagreb, Bijeni\v{c}ka cesta 30, 10000 Zagreb, Croatia}
\email{aleksandar.bulj@math.hr}
\email{vjekovac@math.hr}



\begin{abstract}
We study Fourier multiplier operators associated with symbols $\xi\mapsto \exp(\mathbbm{i}\lambda\phi(\xi/|\xi|))$, where $\lambda$ is a real number and $\phi$ is a real-valued $\textup{C}^\infty$ function on the standard unit sphere $\mathbb{S}^{n-1}\subset\mathbb{R}^n$.
For $1<p<\infty$ we investigate asymptotic behavior of norms of these operators on $\textup{L}^p(\mathbb{R}^n)$ as $|\lambda|\to\infty$.
We show that these norms are always $O((p^\ast-1) |\lambda|^{n|1/p-1/2|})$, where $p^\ast$ is the larger number between $p$ and its conjugate exponent. More substantially, we show that this bound is sharp in all even-dimensional Euclidean spaces $\mathbb{R}^n$.
In particular, this gives a negative answer to a question posed by Maz'ya.
Concrete operators that fall into the studied class are the multipliers forming the two-dimen\-sional Riesz group, given by the symbols 
$r\exp(\mathbbm{i}\varphi) \mapsto \exp(\mathbbm{i}\lambda\cos\varphi)$.
We show that their $\textup{L}^p$ norms are comparable to $(p^\ast-1) |\lambda|^{2|1/p-1/2|}$ for large $|\lambda|$, solving affirmatively a problem suggested in the work of Dragi\v{c}evi\'{c}, Petermichl, and Volberg.
\end{abstract}

\maketitle


\section{Introduction}
Consider Fourier multiplier operators $T_{\phi}^{\lambda}$ associated with symbols of the form
\begin{equation*}
m_{\phi}^{\lambda}(\xi) := e^{\mathbbm{i}\lambda\phi(\xi/|\xi|)}; \quad\xi\in\mathbb{R}^n\setminus\{\mathbf{0}\},
\end{equation*}
i.e., $T_{\phi}^{\lambda}$ acts on Schwartz functions $f$ on the Fourier side as
\[ (\widehat{T_{\phi}^{\lambda}f})(\xi) = m_{\phi}^{\lambda}(\xi) \widehat{f}(\xi). \]
Here, $\phi\in\textup{C}^{\infty}(\mathbb{S}^{n-1})$ is a real-valued \emph{phase function} on the unit sphere, while $\lambda$ is a real parameter. 
By \cite[\S III.3.5, Theorem~6]{Stein70book} we know that $T_{\phi}^{\lambda}$ has a representation 
\begin{equation}\label{eq:Trepresentation}
T_{\phi}^{\lambda}f = a_{\phi}^{\lambda} f + S_{\phi}^{\lambda}f,
\end{equation}
where $a_{\phi}^{\lambda}$ is a constant given by
\begin{equation*}
a_{\phi}^{\lambda} := \frac{1}{\sigma_{n-1}(\mathbb{S}^{n-1})} \int_{\mathbb{S}^{n-1}} m_{\phi}^{\lambda}(\xi) \,\textup{d}\sigma_{n-1}(\xi),
\end{equation*}
while $S_{\phi}^{\lambda}$ is a singular integral operator defined as
\begin{equation}\label{eq:kernel}
(S_{\phi}^{\lambda}f)(x) := \lim_{\varepsilon\to0+} \int_{|y|\geq\varepsilon} f(x-y) \,\frac{\Omega_{\phi}^{\lambda}(y/|y|)}{|y|^n} \,\textup{d}y; \quad x\in\mathbb{R}^n
\end{equation}
for some $\Omega_{\phi}^{\lambda}\in\textup{C}^{\infty}(\mathbb{S}^{n-1})$ with mean zero, i.e., $\int_{\mathbb{S}^{n-1}}\Omega_{\phi}^{\lambda}(y)\,\textup{d}\sigma_{n-1}(y)=0$.
Here $\sigma_{n-1}$ denotes the $(n-1)$-dimensional spherical measure.
In our case we clearly have
\begin{equation}\label{eq:coefficient}
|a_{\phi}^{\lambda}| \leq 1.
\end{equation}
By classical results of the Calder\'{o}n--Zygmund theory (see \cite[\S II.4.2, Theorem~3]{Stein70book} or \cite[Theorem~1]{CZ56}) we know that $T_{\phi}^{\lambda}$ extends to a bounded linear operator on $\textup{L}^p(\mathbb{R}^n)$ for every $p\in(1,\infty)$, and the kernel representation \eqref{eq:Trepresentation}--\eqref{eq:kernel} remains valid for every $f\in\textup{L}^p(\mathbb{R}^n)$.

For each $p\in(1,\infty)$ we thus arrived at a one-parameter group of bounded linear operators $(T_{\phi}^{\lambda})_{\lambda\in\mathbb{R}}$ on the Banach space $\textup{L}^p(\mathbb{R}^n)$.
Plancherel's theorem and unimodularity of $m_{\phi}^{\lambda}$ give 
\begin{equation}\label{eq:Plancherel}
\| T_{\phi}^{\lambda} \|_{\textup{L}^2(\mathbb{R}^n)\to\textup{L}^2(\mathbb{R}^n)} = 1,
\end{equation}
while for $p\neq2$ it makes sense to investigate asymptotic behavior of the $\textup{L}^p$ norms of $T_{\phi}^{\lambda}$ as $|\lambda|\to\infty$.
The present paper is motivated in part by the following question by Vladimir Maz'ya, formulated as Problem~15 on his list of $75$ open problems \cite{Mazya75problems}.

\begin{problem}[from {\cite[Subsection~4.2]{Mazya75problems}}]
\label{prob:Mazya}
Prove or disprove the estimate
\begin{equation}\label{eq:Mazyaconj}
\| T_{\phi}^{\lambda} \|_{\textup{L}^p(\mathbb{R}^n)\to\textup{L}^p(\mathbb{R}^n)} \leq C_{n,p,\phi} |\lambda|^{(n-1)|1/p-1/2|},
\end{equation}
where $|\lambda|\geq1$ and $1<p<\infty$, while the constant $C_{n,p,\phi}$ depends on $n$, $p$, and $\phi$. 
\end{problem}

In Theorem~\ref{thm:main} below we find the largest possible growth in $|\lambda|$ of the $\textup{L}^p$ norms of multiplier operators $T_{\phi}^{\lambda}$ in every even number of dimensions $n$.
It will turn out that the answer to Problem~\ref{prob:Mazya} is negative in all even-dimensional Euclidean spaces $\mathbb{R}^n$.
Moreover, we will also be concerned with sharp dependence of the constant $C_{n,p,\phi}$ on the exponent $p$.

The origins of Problem~\ref{prob:Mazya} trace back to the papers by Maz'ya and Ha\u{\i}kin \cite{MH69,MH76} on rather general multiplier theorems. 
Specific operators $T_{\phi}^{\lambda}$ with the particular phase
\begin{equation}\label{eq:RieszSobolev}
\phi(\xi) = \xi_1; \quad\xi=(\xi_1,\ldots,\xi_n)\in\mathbb{S}^{n-1}
\end{equation}
appear in the analysis of the Navier--Stokes equations; see \cite[\S2]{GIMM06}, \cite[\S4]{GIMM07}, \cite[Eq.~(1.3)]{GS13}, or \cite[Eq.~(23)]{FM12}.
This phase leads to a one-parameter uniformly continuous operator group $(T_{\phi}^{\lambda})_{\lambda\in\mathbb{R}}$ on every $\textup{L}^p(\mathbb{R}^n)$, $1<p<\infty$, called the \emph{Riesz group}.
Its infinitesimal generator is simply the \emph{Riesz transform},
\[ (\widehat{R_1 f})(\xi) = - \mathbbm{i} \frac{\xi_1}{|\xi|} \widehat{f}(\xi); \quad \xi=(\xi_1,\ldots,\xi_n)\in\mathbb{R}^n\setminus\{\mathbf{0}\}. \]
If $n=2$, then \eqref{eq:RieszSobolev} becomes simply the phase $\phi(e^{\mathbbm{i}\varphi}) = \cos\varphi$, which yields the two-dimensional symbol
\begin{equation}\label{eq:cossymbol}
m_{\cos}^{\lambda}(r e^{\mathbbm{i}\varphi}) := e^{\mathbbm{i}\lambda\cos\varphi}; \quad r\in(0,\infty),\ \varphi\in\mathbb{R}.
\end{equation}

One-dimensional case of estimate \eqref{eq:Mazyaconj} is easily seen to hold, as $T_{\phi}^{\lambda}$ is always a bounded linear combination of the identity and the Hilbert transform, so it satisfies $\textup{L}^p$ estimates that are independent of $\lambda$.
In higher dimensions, the H\"{o}rmander-Mihlin theorem (see \cite[Theorem~6.2.7]{Gra14book}) gives a weak $\textup{L}^1$ bound
\[ \| T_{\phi}^{\lambda} \|_{\textup{L}^1(\mathbb{R}^n)\to\textup{L}^{1,\infty}(\mathbb{R}^n)} \leq C_{n,\phi} \,|\lambda|^{\lfloor n/2\rfloor+1}, \] 
which can then be interpolated with the $\textup{L}^2$ identity \eqref{eq:Plancherel} and dualized to deduce
\begin{equation}\label{eq:cheapHMest}
\| T_{\phi}^{\lambda} \|_{\textup{L}^p(\mathbb{R}^n)\to\textup{L}^p(\mathbb{R}^n)} \leq C_{n,\phi,p} \,|\lambda|^{(2\lfloor n/2\rfloor+2)|1/p-1/2|} . 
\end{equation}
This makes one suspect that the sharp exponent of $|\lambda|$ on the average grows by $|1/p-1/2|$ as we increase the number of dimensions $n$ by $1$.
Thus, inequality \eqref{eq:Mazyaconj} is actually a reasonable guess. 

\smallskip
Let us now formulate the main result of this paper.
For every $p\in(1,\infty)$ we denote $p^{\ast}:=\max\{p,p/(p-1)\}$.

\begin{theorem}\label{thm:main}
\begin{enumerate}[(a)]
\item\label{part:thmparta}
Fix an integer $n\geq2$ and a real-valued phase function $\phi\in\textup{C}^{\infty}(\mathbb{S}^{n-1})$.
There is a finite constant $C_{n,\phi}$ such that for every exponent $p\in(1,\infty)$ and every $\lambda\in\mathbb{R}$ satisfying $|\lambda|\geq1$ we have
\begin{equation}\label{eq:mainupper}
\|T_{\phi}^{\lambda}\|_{\textup{L}^p(\mathbb{R}^n)\to\textup{L}^p(\mathbb{R}^n)} \leq C_{n,\phi} \,(p^{\ast}-1) \,|\lambda|^{n|1/p-1/2|}
\end{equation}
and
\begin{equation}\label{eq:weakupper}
\|T_{\phi}^{\lambda}\|_{\textup{L}^1(\mathbb{R}^n)\to\textup{L}^{1,\infty}(\mathbb{R}^n)} \leq C_{n,\phi} \,|\lambda|^{n/2}.
\end{equation}
\item\label{part:thmpartb}
Fix an \underline{even} integer $n\geq2$.
There exist a real-valued phase function $\phi\in\textup{C}^{\infty}(\mathbb{S}^{n-1})$ and a constant $c_{n,\phi}>0$ such that for every exponent $p\in(1,\infty)$ and every nonzero integer $k$ we have
\begin{equation}\label{eq:mainlower}
\|T_{\phi}^{k}\|_{\textup{L}^p(\mathbb{R}^n)\to\textup{L}^p(\mathbb{R}^n)} \geq c_{n,\phi} \,(p^{\ast}-1) \,|k|^{n|1/p-1/2|}
\end{equation}
and
\begin{equation}\label{eq:weaklower}
\|T_{\phi}^{k}\|_{\textup{L}^1(\mathbb{R}^n)\to\textup{L}^{1,\infty}(\mathbb{R}^n)} \geq c_{n,\phi} \,|k|^{n/2}.
\end{equation}
\end{enumerate}
\end{theorem}

In particular, notice that \eqref{eq:mainupper} improves the ``cheap'' bound \eqref{eq:cheapHMest}, while \eqref{eq:mainlower} disproves the conjectured estimate \eqref{eq:Mazyaconj} in all even dimensions $n\geq2$.
Let us remark that \eqref{eq:mainlower} easily extends to non-integer values of $k$ using the group property of the operators $T_{\phi}^{\lambda}$, but at the cost of possibly losing sharp dependence on $p$; cf.\@ the comments in \cite{CDK21}.
Techniques that we use also allow us to obtain weak $\textup{L}^1$ estimates \eqref{eq:weakupper} and \eqref{eq:weaklower}.

While \eqref{eq:weakupper} and \eqref{eq:Plancherel} will immediately imply \eqref{eq:mainupper}, there are also other ways to establish upper $\textup{L}^p$ bounds of that form.
The number $2\lfloor n/2\rfloor+2$ in the exponent on the right hand side of \eqref{eq:cheapHMest} can be easily lowered to anything strictly larger than $n$ by considering more sophisticated versions of the H\"{o}rmander--Mihlin theorem, such as those in \cite{CT77,Seg90,GS19}, but the optimal exponent is trickier.
Shortly after the first preprint of the present paper was made public, Stolyarov \cite{Sto22} showed us an interpolation argument that gives the same sharp exponent in \eqref{eq:mainupper}. However this argument does not seem to give the sharp order of the constant in terms of $p$ and it misses the weak endpoint \eqref{eq:weakupper}.

Moreover, Stolyarov \cite{Sto22} independently showed lower estimates
\[ \|T_{\phi}^{\lambda}\|_{\textup{L}^p(\mathbb{R}^n)\to\textup{L}^p(\mathbb{R}^n)} \geq c_{n,\phi,p} \,|\lambda|^{n|1/p-1/2|} \]
for a particular choice of the phase $\phi$ in both even and odd dimensions $n\geq2$ using different techniques from ours. 
Just as for the upper bound, we do not see how to modify his approach to give sharp dependence on $p$ of the constant $c_{n,\phi,p}$ in the above lower bound.

\smallskip
Two-dimensional multiplier operators $T_{\cos}^{\lambda}$ with very concrete symbols \eqref{eq:cossymbol} were already studied by Dragi\v{c}evi\'{c}, Petermichl, and Volberg in \cite{DPV06}. Their paper, which might have been overlooked in \cite{Mazya75problems}, claims bounds of the form
\begin{equation}\label{eq:DPVclaim}
c_{\delta} \,(p^{\ast}-1) \,|k|^{2|1/p-1/2|-\delta} \leq \|T_{\cos}^{k}\|_{\textup{L}^p(\mathbb{R}^2)\to\textup{L}^p(\mathbb{R}^2)} \leq C  \,(p^{\ast}-1) \,|k|^{2|1/p-1/2|}
\end{equation}
for every $\delta>0$, every $p\in(1,\infty)$, and every nonzero integer $k$.
The lower bound in \eqref{eq:DPVclaim} is sketched in the proof of \cite[Theorem~6]{DPV06} and it already disproves the estimate \eqref{eq:Mazyaconj} in $n=2$ dimensions.
Since the authors of \cite{DPV06} remark that they ``do not know how to get rid of $\delta$'' in \eqref{eq:DPVclaim}, an optimal growth of the $\textup{L}^p$ norms of $T_{\cos}^{k}$ is an interesting separate problem, which we fully address in the following theorem.

\begin{theorem}\label{thm:cos}
Let $T_{\cos}^{\lambda}$ be the Fourier multiplier operator associated with the symbol \eqref{eq:cossymbol}.
There exist constants $0<c\leq C<\infty$ such that for every exponent $p\in(1,\infty)$ and every $\lambda\in\mathbb{R}$ satisfying $|\lambda|\geq1$ we have
\begin{equation}\label{eq:cosstrong}
c \,(p^{\ast}-1) \,|\lambda|^{2|1/p-1/2|} \leq \|T_{\cos}^{\lambda}\|_{\textup{L}^p(\mathbb{R}^2)\to\textup{L}^p(\mathbb{R}^2)} \leq C  \,(p^{\ast}-1) \,|\lambda|^{2|1/p-1/2|}
\end{equation}
and
\begin{equation}\label{eq:cosweak}
c \,|\lambda| \leq \|T_{\cos}^{\lambda}\|_{\textup{L}^1(\mathbb{R}^2)\to\textup{L}^{1,\infty}(\mathbb{R}^2)} \leq C \,|\lambda|.
\end{equation}
\end{theorem}

\smallskip
The upper estimates in Theorem~\ref{thm:main}, and thus also those in Theorem~\ref{thm:cos}, will be established in Section~\ref{sec:upper}. We use weak $\textup{L}^1$ estimates for singular integrals in terms of the size of the kernel alone and no smoothness assumptions imposed; see the series of papers by Christ and Rubio de Francia \cite{CR88}, Hofmann \cite{Hof88}, Seeger \cite{Seeger96}, and Tao \cite{Tao99}.
That way we only need to bound $\|\Omega_{\phi}^{\lambda}\|_{\textup{L}^2(\mathbb{S}^{n-1})}$ for the singular kernel appearing in the singular integral part \eqref{eq:kernel}. That is achieved by generalizing the two-dimensional approach of Dragi\v{c}evi\'{c}, Petermichl, and Volberg \cite[Theorem~5]{DPV06} to higher dimensions: replacing the Fourier series expansion by the expansion into spherical harmonics, and replacing one-dimensional derivatives with powers of the spherical Laplacean.

\smallskip
The lower estimates for the $\textup{L}^p$ norms of operators $T_{\phi}^{\lambda}$ are more substantial results of this paper. To some extent we generalize an approach by Carbonaro, Dragi\v{c}evi\'{c}, and one of the present authors \cite{CDK21}. That paper was only concerned with asymptotics for powers of a particular two-dimensional Fourier multiplier with the complex symbol $\xi\mapsto\xi/|\xi|$.
Here we develop a convenient way of bounding $\textup{L}^p$ norms of more general Fourier multipliers from below by merely choosing two particular spherical functions, $u$ and $v$, with mutually related expansions into spherical harmonics.
Let us already state the result, referring the reader to Subsection~\ref{subsec:spherical} for a review of spherical harmonics.
For an integer $j\geq0$ and a real parameter $\alpha\in[0,n]$ denote the constants
\begin{equation}\label{eq:defcoeffgamma}
\gamma_{n,j,\alpha} := \pi^{n/2-\alpha} \frac{\Gamma((j+\alpha)/2)}{\Gamma((j+n-\alpha)/2)}.
\end{equation}

\begin{theorem}\label{thm:generallower}
Let $p\in[1,2]$ and $q\in[2,\infty]$ be mutually conjugate exponents and let $m$ be a bounded homogeneous Borel-measurable symbol on $\mathbb{R}^n\setminus\{\mathbf{0}\}$. Take a sequence $(Y_j)_{j=0}^{\infty}$ such that:
\begin{enumerate}[(a)]
\item\label{it:propa} 
each $Y_j$ is from the linear space of spherical harmonics on $\mathbb{S}^{n-1}$ of degree $j$;
\item\label{it:propb}  
the series $\sum_{j=0}^{\infty}Y_j$ converges in $\textup{L}^{q}(\mathbb{S}^{n-1})$ to some function $u$;
\item\label{it:propc}  
the orthogonal series $\sum_{j=0}^{\infty}\mathbbm{i}^{-j}\gamma_{n,j,n/p}Y_j$ converges in $\textup{L}^2(\mathbb{S}^{n-1})$ to some function $v$.
\end{enumerate}
If $p>1$, $q<\infty$, then the Fourier multiplier operator $T_m$ associated with $m$ satisfies the bound
\begin{align}
\| T_m \|_{\textup{L}^p(\mathbb{R}^n)\to\textup{L}^p(\mathbb{R}^n)}
& \geq \frac{\gamma_{n,0,n/q}}{\sigma(\mathbb{S}^{n-1})^{1/p}} \frac{|\langle m, v \rangle_{\textup{L}^2(\mathbb{S}^{n-1})}|}{\|u\|_{\textup{L}^{q}(\mathbb{S}^{n-1})}} \label{eq:propmainlowsh} \\
& \geq c_n \,(q-1) \,\frac{|\langle m, v \rangle_{\textup{L}^2(\mathbb{S}^{n-1})}|}{\|u\|_{\textup{L}^{q}(\mathbb{S}^{n-1})}}, \label{eq:propmainlow}
\end{align}
while in the endpoint case $p=1$, $q=\infty$ we have
\begin{equation}\label{eq:propweaklow}
\| T_m \|_{\textup{L}^1(\mathbb{R}^n)\to\textup{L}^{1,\infty}(\mathbb{R}^n)} \geq \frac{c}{n}\, \frac{|\langle m, v \rangle_{\textup{L}^2(\mathbb{S}^{n-1})}|}{\|u\|_{\textup{L}^{\infty}(\mathbb{S}^{n-1})}} .
\end{equation}
Here, $c_n>0$ is a constant depending on $n$, while $c>0$ is an absolute constant.
\end{theorem}

Theorem~\ref{thm:generallower} combined with some guessing of appropriate spherical functions $u$ and $v$ is a useful tool for proving lower bounds for multipliers with homogeneous unimodular symbols.
In particular, it will be a crucial ingredient in the proof of the part \eqref{part:thmpartb} of Theorem~\ref{thm:main} and in the proof of Theorem~\ref{thm:cos}; see Sections~\ref{sec:thm2lower} and \ref{sec:proofofcos}, respectively.
The proof of Theorem~\ref{thm:generallower} in Section~\ref{sec:lower} will, in turn, build on the approach from \cite[Section~6]{CDK21}, but with additional complications arising from arbitrary symbols and higher dimensions.

Estimate \eqref{eq:propmainlowsh} is tailored to exact constants and we believe that it is, in fact, absolutely sharp for many multipliers.
For instance, if one considers the two-dimensional complex symbol $m(\xi)=\overline{\xi}/\xi$, $\xi\in\mathbb{C}$, then the underlying operator $T_m$ is the Ahlfors--Beurling operator. By choosing $u=m$ on $\mathbb{S}^1$, the inequality \eqref{eq:propmainlowsh} simplifies to $\| T_m \|_{\textup{L}^p(\mathbb{C})\to\textup{L}^p(\mathbb{C})} \geq q-1$, which reproves the result of Lehto \cite{Leh65} and matches the well-known conjecture by Iwaniec \cite{Iwa82} on the exact $\textup{L}^p$ norm of $T_m$.
Estimate \eqref{eq:propmainlowsh} is also believed to be sharp in the case of $m(\xi)=(\xi/|\xi|)^k$, $\xi\in\mathbb{C}$, for an integer $k$; see the paper \cite{CDK21} as this estimate generalizes \cite[Theorem~6.1]{CDK21}.
This potential sharpness of Theorem~\ref{thm:generallower} can be viewed both as a virtue and as a lack of flexibility, by focusing on global and not local properties of the multiplier.
In particular we do not see how to use that theorem to prove lower $\textup{L}^p$ bounds for $T_{\phi}^{\lambda}$ that are simultaneously sharp in $\lambda$ and $p$ in odd dimensions $n\geq3$.


\section{Preliminaries}

\subsection{Notation and terminology}

We use the following variants of the Hardy--Vinogradov and the Bachmann--Landau notations.
Let $A$ and $B$ be two complex functions on a set $X$. 
We write 
\[ A(x) \lesssim_P B(x) \quad\text{and}\quad B(x) \gtrsim_P A(x) \]
if the inequality $|A(x)|\leq C_P |B(x)|$ holds for every $x\in X$, with some finite constant $C_P$ depending on a set of parameters $P$.
Moreover,
\[ A(x) \sim_P B(x) \]
if both $A(x) \lesssim_P B(x)$ and $B(x) \lesssim_P A(x)$ hold.
Next, assume that $A$ and $B$ are, more specifically, complex functions of a single real (or complex) variable $x$ and that $a\in\mathbb{R}\cup\{-\infty,\infty\}$ (or $a\in\mathbb{C}\cup\{\infty\}$) is a fixed point.
We write
\[ A(x) = O_P^{x\to a}\big(B(x)\big) \]
if $\limsup_{x\to a}|A(x)/B(x)|<\infty$ and
\[ A(x) = o_P^{x\to a}\big(B(x)\big) \]
if $\lim_{x\to a}A(x)/B(x)=0$.
Here $P$ in the subscript emphasizes that $A$ and $B$ are also allowed to depend on the parameters from $P$, but the limits need not be uniform in those parameters.

The \emph{Euclidean norm} (i.e., the \emph{$\ell^2$ norm}) on $\mathbb{R}^n$ will be written simply as $x\mapsto|x|$, while the \emph{dot product} (i.e., the standard inner product) of $x,y\in\mathbb{R}^n$ is denoted $x\cdot y$.
The \emph{standard unit sphere} in $\mathbb{R}^n$ is
\[ \mathbb{S}^{n-1} := \{x\in\mathbb{R}^n : |x|=1 \}. \]
The \emph{surface measure} on $\mathbb{S}^{n-1}$, i.e., the $(n-1)$-dimensional \emph{spherical measure}, is the restriction of the $(n-1)$-dimensional Hausdorff measure to Borel subsets of $\mathbb{S}^{n-1}$; it is written as $\sigma_{n-1}$.
The notation for the measure is suppressed whenever the integrals are evaluated with respect to the Lebesgue measure on $\mathbb{R}^n$.
The \emph{Lebesgue norms} $\|\cdot\|_{\textup{L}^p}$ and the \emph{Lebesgue spaces} $\textup{L}^p$ are defined in a standard way; see \cite{Gra14book,Stein70book,SW71Book}. We often denote the measure space in the parentheses, such as $\textup{L}^p(\mathbb{S}^{n-1})$, the underlying measure being understood.
Sometimes, we write the variable in which the $\textup{L}^p$ norm is taken in the subscript, such as $\|f(x,y)\|_{\textup{L}^p_x}$. 

The \emph{imaginary unit} will be denoted by $\mathbbm{i}$.
Any \emph{logarithm} is having $e$ as its base.
We use the notation $\mathbbm{1}_A$ for the \emph{indicator function} (i.e., the \emph{characteristic function}) of a set $A$.
For $x\in\mathbb{R}$ we respectively write $\lfloor x\rfloor$ and $\lceil x\rceil$ for the largest integer $k$ such that $k\leq x$ and the smallest integer $l$ such that $l\geq x$.
It $T\colon X\to Y$ is a linear operator between normed spaces $(X,\|\cdot\|_X)$ and $(Y,\|\cdot\|_Y)$, then we write $\|T\|_{X\to Y}$ for its \emph{operator norm}, defined as
\[ \|T\|_{X\to Y} := \sup_{x\in X,\,\|x\|_X=1} \|Tx\|_{Y}. \]
We say that $p,q\in[1,\infty]$ are \emph{conjugated exponents} if $1/p+1/q=1$ holds.
We have already been writing $p^\ast$ for the larger number between $p\in[1,\infty]$ and its conjugate exponent.

The \emph{Fourier transform} is initially defined for a function $f\in\textup{L}^1(\mathbb{R}^n)$ as $\widehat{f}\colon\mathbb{R}^n\to\mathbb{C}$ given by the formula
\[ \widehat{f}(\xi) := \int_{\mathbb{R}^n} f(x) e^{-2\pi\mathbbm{i} x\cdot\xi} \,\textup{d}x; \quad \xi\in\mathbb{R}^n; \]
note that we are using the widespread normalization from \cite{SW71Book}.
It extends to a unitary operator on $\textup{L}^2(\mathbb{R}^n)$ and also to bounded linear maps from $\textup{L}^p(\mathbb{R}^n)$ to $\textup{L}^q(\mathbb{R}^n)$ for any pair of conjugated exponents $p\in[1,2]$ and $q\in[2,\infty]$.
Moreover, it extends to the space of tempered distributions $F$ via the duality formula: $\widehat{F}(\varphi) := F(\widehat{\varphi})$ for every Schwartz function $\varphi$.

A function $f\colon\mathbb{R}^n\setminus\{0\}\to\mathbb{C}$ is \emph{homogeneous of degree $j$} if
$f(t x) = t^j f(x)$
holds for every $t>0$ and every $x\in\mathbb{R}^n$.
It is simply said to be \emph{homogeneous} if it is homogeneous of degree $0$.
Thus, a polynomial $P$ of $n$ real variables $x_1,\ldots,x_n$ is homogeneous of degree $j$ precisely when it is a linear combination of the monomials $x_1^{k_1}\cdots x_n^{k_n}$ for nonnegative integers $k_1,\ldots,k_n$ adding up to $j$.

\subsection{Properties of spherical harmonics}
\label{subsec:spherical}
For reader's convenience, in this subsection we review several results on spherical harmonics that will be needed later.
Basic properties are taken from the book by Stein and Weiss \cite[Sections~IV.2--IV.4]{SW71Book} and the book by Stein \cite[Section~III.3]{Stein70book}. For more sophisticated $\textup{L}^p$ estimates concerning spherical harmonics we will recall the work of Sogge \cite{Sogge85,Sogge86}.

Throughout the paper we are working in $\mathbb{R}^n$ for a fixed dimension $n\geq2$. Let us also take a nonnegative integer $j$.
Homogeneous polynomials in $n$ variables of degree $j$ that are also harmonic functions on $\mathbb{R}^n$ (i.e., satisfy the Laplace equation) are called \emph{solid spherical harmonics of degree $j$}.
Their restrictions to the sphere $\mathbb{S}^{n-1}$ are called \emph{(surface) spherical harmonics of degree $j$}.
Spherical harmonics of distinct degrees are mutually orthogonal and the whole space $\textup{L}^2(\mathbb{S}^{n-1})$ can be written as an orthogonal sum of (finite dimensional) spaces of spherical harmonics of degrees $j=0,1,2,\ldots$; see \cite[Chapter~IV, Corollaries~2.3 and 2.4]{SW71Book}.

Spherical harmonics play important roles in describing how the Fourier transform acts on many particular types of functions and distributions.
If $P$ is a solid spherical harmonic of degree $j$, then
\begin{equation}\label{eq:FT_on_Gaussians}
f(x) = P(x) e^{-\pi|x|^2} \ \implies\ \widehat{f}(\xi) = \mathbbm{i}^{-j} P(\xi) e^{-\pi|\xi|^2},
\end{equation}
by \cite[Chapter~IV, Theorem~3.4]{SW71Book}.
The relevance of constants \eqref{eq:defcoeffgamma} comes from a formula by Bochner \cite{Boc51} (also see
\cite[Chapter~IV, Theorem~4.1]{SW71Book}): if $Y$ is a spherical harmonic of degree $j$ and if $0<\alpha<n$, then
\begin{equation}\label{eq:FT_on_Riesz}
K(x) = Y\Big(\frac{x}{|x|}\Big) |x|^{-n+\alpha} \ \implies\  \widehat{K}(\xi) = \mathbbm{i}^{-j} \gamma_{n,j,\alpha} Y\Big(\frac{\xi}{|\xi|}\Big) |\xi|^{-\alpha}.
\end{equation}
However, the last function $K$ is only locally integrable, so the Fourier transform needs to be understood as acting on the space of tempered distributions. 
Bochner's formula \eqref{eq:FT_on_Riesz} also holds in the limiting case $\alpha=0$ if $j\geq1$ and the function $K$ is interpreted as a principal value distribution 
\[ f \ \mapsto\ \operatorname{p.v.} \int_{\mathbb{R}^n} K(x) f(x) \,\textup{d}x = \lim_{\varepsilon\to0+} \int_{\{|x|\geq\varepsilon\}} K(x) f(x) \,\textup{d}x . \]
Then it reads
\begin{equation}\label{eq:FT_on_kernel}
K(x) = \operatorname{p.v.} Y\Big(\frac{x}{|x|}\Big) |x|^{-n} \ \implies\  \widehat{K}(\xi) = \mathbbm{i}^{-j} \gamma_{n,j,0} Y\Big(\frac{\xi}{|\xi|}\Big);
\end{equation}
see \cite[Chapter~IV, Theorem~4.5]{SW71Book}.
Stein and Weiss also formulated the ultimate consequence of \eqref{eq:FT_on_kernel} as \cite[Chapter~IV, Theorem~4.7]{SW71Book}: if $\Omega,\Omega_0\in\textup{L}^2(\mathbb{S}^{n-1})$ have related expansions into spherical harmonics of the form
\[ \Omega = \sum_{j=1}^{\infty} Y_j, \quad \Omega_0 = \sum_{j=1}^{\infty} \mathbbm{i}^{-j} \gamma_{n,j,0} Y_j, \]
then
\begin{equation}\label{eq:FT_general_SW}
K(x) = \operatorname{p.v.} \Omega\Big(\frac{x}{|x|}\Big) |x|^{-n} \ \implies\  \widehat{K}(\xi) = \Omega_0\Big(\frac{\xi}{|\xi|}\Big).
\end{equation}

Observe that
\begin{equation}\label{eq:gammasymmetry}
\gamma_{n,j,n-\alpha} = \frac{1}{\gamma_{n,j,\alpha}}
\end{equation}
holds whenever the constants $\gamma_{n,j,\alpha}$ are defined.
By writing
\[ \log\frac{\Gamma(j/2+\alpha/2)}{(j/2)^{\alpha/2}\Gamma(j/2)} 
= \int_{0}^{\alpha/2} \Big(\psi\Big(\frac{j}{2}+t\Big)-\log\frac{j}{2}\Big) \,\textup{d}t \]
in terms of the digamma function $\psi=\Gamma'/\Gamma$ and using the asymptotic expansion of $\psi$, see \cite[Eq.~5.11.2]{NIST} or \cite[Eq.~6.3.18]{ASbook64}, we easily conclude
\begin{equation}\label{eq:gammaasympt}
\gamma_{n,j,\alpha} \sim_{n} j^{\alpha-n/2}
\end{equation}
for $\alpha\in[0,n]$ and a positive integer $j$.
Also, writing
\[ \Gamma\Big(\frac{\alpha}{2}\Big) = \frac{\Gamma(\alpha/2 + 1)}{\alpha/2}, \quad \Gamma\Big(\frac{n-\alpha}{2}\Big) = \frac{\Gamma((n-\alpha)/2 + 1)}{(n-\alpha)/2}, \]
we easily get
\begin{equation}\label{eq:gammaasympt2}
\gamma_{n,0,\alpha} \sim_{n} \frac{n}{\alpha}-1
\end{equation}
for every $\alpha\in(0,n)$.

Let us also recall the \emph{spherical Laplacean}, which is a particular case of the \emph{Laplace--Beltrami operator}. In the case of the sphere $\mathbb{S}^{n-1}$ we can define $\Delta_{\mathbb{S}^{n-1}}f$ for a $\textup{C}^2$ function $f\colon\mathbb{S}^{n-1}\to\mathbb{C}$ by applying the ordinary $n$-dimensional Laplace operator $\Delta=\Delta_{\mathbb{R}^n}$ to the homogeneous function $\mathbb{R}^n\setminus\{\mathbf{0}\}\to\mathbb{C}$, $x\mapsto f(x/|x|)$ and then restricting back to the sphere $\mathbb{S}^{n-1}$. 
Spherical harmonics are eigenfunctions of $\Delta_{\mathbb{S}^{n-1}}$. More precisely, if $Y_j$ is a spherical harmonic of degree $j$, then
\begin{equation}\label{eq:eigenfunctions}
\Delta_{\mathbb{S}^{n-1}} Y_j = - j (j+n-2) Y_j ;
\end{equation}
see \cite[\S III.3.1.4]{Stein70book}.

For an integer $j\geq0$ let $H_j$ denote the orthogonal projection onto the linear subspace of $\textup{L}^2(\mathbb{S}^{n-1})$ consisting of spherical harmonics of degree $j$ (including the zero-function). Sogge \cite{Sogge86} established the sharp estimate
\[ \| H_j f \|_{\textup{L}^2(\mathbb{S}^{n-1})} \lesssim_{n,p} j^{\tau(n,p)} \| f \|_{\textup{L}^p(\mathbb{S}^{n-1})} \]
for $j\geq1$ and $1\leq p\leq 2$, where the exponent $\tau(n,p)$ is defined as
\[ \tau(n,p) := \begin{cases}
(n-1)\big(\frac{1}{p}-\frac{1}{2}\big)-\frac{1}{2} & \text{for } 1\leq p\leq \frac{2n}{n+2} , \\
\frac{1}{2}(n-2)\big(\frac{1}{p}-\frac{1}{2}\big) & \text{for } \frac{2n}{n+2}\leq p\leq2 .
\end{cases} \]
Since $H_j$ is self-adjoint, the last estimate has its dual formulation:
\[ \| H_j f \|_{\textup{L}^q(\mathbb{S}^{n-1})} \lesssim_{n,p} j^{\tau(n,p)} \| f \|_{\textup{L}^2(\mathbb{S}^{n-1})} , \]
where $2\leq q\leq\infty$ is the conjugate exponent of $p$.
In particular,
\begin{equation}\label{eq:Sogge_est}
 \| Y \|_{\textup{L}^{q}(\mathbb{S}^{n-1})} \lesssim_{n,p} j^{\tau(n,p)} \| Y \|_{\textup{L}^2(\mathbb{S}^{n-1})} 
\end{equation}
for every spherical harmonic $Y$ of degree $j\geq1$.


\section{Proof of the upper bounds in Theorem~\ref{thm:main}}
\label{sec:upper}
Before we start discussing any proofs, let us give a brief remark on symmetries of $T_{\phi}^{\lambda}$, which needs to be kept in mind throughout the paper.

\begin{remark}\label{rem:duality}
In the proofs of any upper or lower $\textup{L}^p$ bounds for $T_{\phi}^{\lambda}$ we can focus on the case $\lambda\geq1$ and $p\leq2$ only. This fact is an immediate consequence of the duality of $\textup{L}^p$ spaces and 
\[ \langle T_{\phi}^{\lambda} f, g \rangle_{\textup{L}^2(\mathbb{R}^n)} = \langle f, T_{\phi}^{-\lambda} g \rangle_{\textup{L}^2(\mathbb{R}^n)} = \langle T_{\phi}^{\lambda} \widetilde{g}, \widetilde{f} \rangle_{\textup{L}^2(\mathbb{R}^n)}, \]
where $\widetilde{f}(x)=\overline{f(-x)}$, $\widetilde{g}(x)=\overline{g(-x)}$.
\end{remark}

\smallskip
The upper bound \eqref{eq:weakupper} in Theorem~\ref{thm:main} is reduced to the weak $\textup{L}^1$ bound for the singular integral given in \eqref{eq:kernel}, by using representation \eqref{eq:Trepresentation} and an obvious bound \eqref{eq:coefficient}.
Estimate \eqref{eq:mainupper} then follows from the Marcinkiewicz interpolation theorem \cite[Theorem~1.3.2]{Gra14book}, which interpolates between the endpoint $\textup{L}^1$ case and the trivial $\textup{L}^2$ case \eqref{eq:Plancherel}, followed by duality observations in Remark~\ref{rem:duality}. 

Weak estimates for singular integrals with rough kernels were first proved in $n=2$ dimensions independently by Christ and Rubio de Francia \cite{CR88} and Hofmann \cite{Hof88}. A higher-dimensional analogue was later shown by Seeger \cite{Seeger96} and subsequently generalized further by Tao \cite{Tao99}. We will use the following theorem to establish \eqref{eq:weakupper}.

\begin{theorem}[from \cite{Seeger96}]
\label{thm:Seeger}
Let $\Omega\in \textup{L}^1(\mathbb{S}^{n-1})$ be such that $\int_{\mathbb{S}^{n-1}}\Omega(x)\,\textup{d}\sigma_{n-1}(x)=0$. If we denote $K(x)=\Omega(x/|x|)|x|^{-n}$, then the operator $S_{\Omega}$ defined for $f\in \textup{C}_{\textup{c}}^{\infty}(\mathbb{R}^n)$ as
\[ (S_{\Omega}f)(x):= \operatorname{p.v.}\int_{\mathbb{R}^n}f(x-y)K(y)\,\textup{d}y  \]
satisfies the bound
\[ \|S_{\Omega}\|_{\textup{L}^{1}(\mathbb{R}^n)\to \textup{L}^{1,\infty}(\mathbb{R}^n)}\lesssim_n 1+ \big\|\widehat{K}\big\|_{\textup{L}^{\infty}(\mathbb{R}^n)} + \int_{\mathbb{S}^{n-1}}|\Omega(x)|\bigg(1+\log_{+}\Big(\frac{|\Omega(x)|}{\|\Omega\|_{\textup{L}^1(\mathbb{S}^{n-1})}}\Big)\bigg) \,\textup{d}\sigma_{n-1}(x). \]
\end{theorem}

Since the last integral is difficult to compute for a kernel that is defined implicitly via the corresponding multiplier symbol, we find it  convenient that the whole expression on the right hand side is bounded by $\|\Omega\|_{\textup{L}^2(\mathbb{S}^{n-1})}$.
Indeed, if we define 
\[ A_0:=\Big\{x\in \mathbb{S}^{n-1}: \frac{|\Omega(x)|}{\|\Omega\|_{\textup{L}^1}} \leq 1 \Big\} 
\quad \text{and} \quad
A_k:=\Big\{x\in \mathbb{S}^{n-1}: 2^{k-1}< \frac{|\Omega(x)|}{\|\Omega\|_{\textup{L}^1}} \leq 2^k\Big\}; \quad k\geq1, \] 
then Chebyshev's inequality implies $\sigma_{n-1}(A_k) \lesssim 2^{-k}$. 
Thus, bounding the logarithm with the upper bound of the function $\Omega$ on the set $A_k$ and
using the Cauchy--Schwarz inequality, it follows that
\begin{align*}
& \int_{\mathbb{S}^{n-1}}|\Omega(x)|\bigg(1+\log_{+}\Big(\frac{|\Omega(x)|}{\|\Omega\|_{\textup{L}^1(\mathbb{S}^{n-1})}}\Big)\bigg) \,\textup{d}\sigma_{n-1}(x)
\lesssim \sum_{k=0}^{\infty} \int_{A_k}|\Omega(x)| (k+1) \,\textup{d}\sigma_{n-1}(x) \\
& \leq \sum_{k=0}^{\infty}(k+1) \|\Omega\|_{\textup{L}^2(\mathbb{S}^{n-1})}\sigma_{n-1}(A_k)^{1/2}
\lesssim \|\Omega\|_{\textup{L}^2(\mathbb{S}^{n-1})}\sum_{k=0}^{\infty}(k+1)2^{-k/2} 
\lesssim \|\Omega\|_{\textup{L}^2(\mathbb{S}^{n-1})}.
\end{align*}
Therefore, in order to apply Theorem~\ref{thm:Seeger} to the operator kernel $\Omega_{\phi}^{\lambda}$ from \eqref{eq:kernel}, we also observe that for $K_{\phi}^{\lambda}(x)=\operatorname{p.v.}\Omega_{\phi}^{\lambda}(x/|x|)|x|^{-n}$ we have
$\big|\widehat{K_{\phi}^{\lambda}}(\xi)\big| =|m_{\phi}^{\lambda}(\xi)- a_{\phi}^{\lambda}| \leq 2$,
where we recall \eqref{eq:coefficient}.
Thus, it remains to prove
\begin{equation}\label{eq:main_upper_ab}
\|\Omega_{\phi}^{\lambda}\|_{\textup{L}^2(\mathbb{S}^{n-1})}\lesssim_{n,\phi} \lambda^{n/2}
\end{equation}
for $\lambda\geq1$.

From equation \eqref{eq:FT_general_SW} and the symmetry property \eqref{eq:gammasymmetry} we can see that expansions of $\Omega_{\phi}^{\lambda}$ and $m_{\phi}^{\lambda}$ into spherical harmonics are related as:
\begin{equation*}
m_{\phi}^{\lambda} = \sum_{j=0}^{\infty} Y_j \quad\implies\quad \Omega_{\phi}^{\lambda} = \sum_{j=1}^{\infty} \mathbbm{i}^j \gamma_{n,j,n} Y_j,
\end{equation*}
where the two series converge in the $\textup{L}^2$ sense.  
Now, the aforementioned asymptotics \eqref{eq:gammaasympt} implies $\gamma_{n,j,n} \sim_{n} j^{n/2}$, so
\begin{equation}\label{eq:omega_upper_first}
\|\Omega_{\phi}^{\lambda}\|_{\textup{L}^2(\mathbb{S}^{n-1})}^2 = \sum_{j = 1}^{\infty} \gamma_{n,j,n}^{2}\|Y_j\|_{\textup{L}^2(\mathbb{S}^{n-1})}^2 \lesssim_{n} \sum_{j=0}^{\infty} j^{n}\|Y_j\|_{\textup{L}^2(\mathbb{S}^{n-1})}^2.
\end{equation}
Recalling that the spherical harmonics are eigenfunctions of the spherical Laplacean, namely that \eqref{eq:eigenfunctions} holds, we arrive at
\[ \Delta_{\mathbb{S}^{n-1}}^rm_{\phi}^{\lambda} = \sum_{j=1}^{\infty} \big(\!-j(j+n-2)\big)^{r} Y_j,\]
for any positive integer $r$. Uniform convergence of the above series is needed for justification of the performed term-by-term differentiation, but it is, in turn, guaranteed by the smoothness of $m_{\phi}^{\lambda}$, Sogge's estimate \eqref{eq:Sogge_est}, and the standard results on rapid convergence of spherical harmonic expansions of smooth functions in \cite[\S~3.1.5]{Stein70book}.
Using \eqref{eq:omega_upper_first} and H\"{o}lder's inequality, and by choosing $r= \lceil n/4 \rceil$ in the previous display, we can write
{\allowdisplaybreaks
\begin{align}
\|\Omega_{\phi}^{\lambda}\|_{\textup{L}^2(\mathbb{S}^{n-1})}^2 
& \lesssim_{n} \Big(\sum_{j=1}^{\infty} j^{4r}\|Y_j\|_{\textup{L}^2(\mathbb{S}^{n-1})}^2 \Big)^{n/4r}\Big(\sum_{j=0}^{\infty}\|Y_j\|_{\textup{L}^2(\mathbb{S}^{n-1})}^2 \Big)^{1-n/4r} \nonumber \\
& \leq \Big(\sum_{j=1}^{\infty} \big(j(j+n-2)\big)^{2r}\|Y_j\|_{\textup{L}^2(\mathbb{S}^{n-1})}^2 \Big)^{n/4r}\Big(\sum_{j=0}^{\infty}\|Y_j\|_{\textup{L}^2(\mathbb{S}^{n-1})}^2 \Big)^{1-n/4r} \nonumber \\
& = \|\Delta_{\mathbb{S}^{n-1}}^rm_{\phi}^{\lambda}\|_{\textup{L}^2(\mathbb{S}^{n-1})}^{n/2r} \|m_{\phi}^{\lambda}\|_{\textup{L}^2(\mathbb{S}^{n-1})}^{2-n/2r}
= \|\Delta_{\mathbb{S}^{n-1}}^rm_{\phi}^{\lambda}\|_{\textup{L}^2(\mathbb{S}^{n-1})}^{n/2r} \nonumber \\
& \lesssim_{n} \|\Delta_{\mathbb{S}^{n-1}}^rm_{\phi}^{\lambda}\|_{\textup{L}^{\infty}(\mathbb{S}^{n-1})}^{n/2r} . \label{eq:omega_upper_second}
\end{align}
}

Therefore, it remains to bound 
\[ \|\Delta_{\mathbb{S}^{n-1}}^r e^{\mathbbm{i}\lambda\phi}\|_{\textup{L}^{\infty}(\mathbb{S}^{n-1})}. \]
First observe that, if a function $f\in \textup{C}^{\infty}(\mathbb{S}^{n-1})$ is of the form 
\[ f=e^{\mathbbm{i}\lambda \phi}\sum_{l=0}^{k}\lambda^l\phi_l, \]
for some nonnegative integer $k$ and some functions $\phi, \phi_0,\dots, \phi_k\in \textup{C}^{\infty}(\mathbb{S}^{n-1}),$  
then there exist functions $\widetilde{\phi}_0,\dots, \widetilde{\phi}_{k+2}\in \textup{C}^{\infty}(\mathbb{S}^{n-1})$ such that
\[ \Delta_{\mathbb{S}^{n-1}}f=e^{\mathbbm{i}\lambda \phi}\sum_{l=0}^{k+2}\lambda^l\widetilde{\phi}_l. \]
The claim is easily seen by calculating the Laplacean $\Delta_{\mathbb{R}^n}$ of the homogenized expression
\[ \mathbb{R}^n\setminus\{0\}\ni x\mapsto e^{\mathbbm{i}\lambda \phi(x/|x|)}\sum_{l=0}^{k}\lambda^{l}\phi_l\Big(\frac{x}{|x|} \Big) \]
and evaluating it at $x\in \mathbb{S}^{n-1}$.
Applying the previous observation inductively, we conclude that there exist functions $\phi_0,\dots, \phi_{2r}\in \textup{C}^{\infty}(\mathbb{S}^{n-1})$ depending only on $\phi$ such that
\[\Delta_{\mathbb{S}^{n-1}}^re^{\mathbbm{i}\lambda\phi}=e^{\mathbbm{i}\lambda\phi}\sum_{l=0}^{2r}\lambda^{l}\phi_{l}.\]
Since $\mathbb{S}^{n-1}$ is compact and the functions $\phi_l$ are continuous, they are also bounded and, therefore, 
\[ \|\Delta_{\mathbb{S}^{n-1}}^r e^{\mathbbm{i}\lambda\phi}\|_{\textup{L}^{\infty}(\mathbb{S}^{n-1})}\leq \sum_{l=0}^{2r}\lambda^{l}\|\phi_l\|_{\textup{L}^{\infty}(\mathbb{S}^{n-1})}\lesssim_{n,\phi} \lambda^{2r} \]
for $\lambda\geq 1$.
The desired estimate \eqref{eq:main_upper_ab} now follows from \eqref{eq:omega_upper_second}.


\section{Proof of Theorem~\ref{thm:generallower}}
\label{sec:lower}
In this section we develop a somewhat general scheme of bounding norms of Fourier multiplier operators from below by constructing functions on $\mathbb{S}^{n-1}$ from infinite sums of spherical harmonics.
The following auxiliary lemma can be thought of as a quantitative refinement of the classical formula \eqref{eq:FT_on_Riesz}.
Recall the constants \eqref{eq:defcoeffgamma}.

\begin{lemma}\label{lm:lowercomp}
For $p\in(1,\infty)$, $q=p/(p-1)$, $\varepsilon\in(0,1/2]$, and a spherical harmonic $Y$ of degree $j\geq0$ one can find a Schwartz function $g_{n,p,\varepsilon,Y}$ such that
\begin{equation}\label{eq:lowercomp1}
\Big\| g_{n,p,\varepsilon,Y}(x) - Y\Big(\frac{x}{|x|}\Big) |x|^{-n/p} \mathbbm{1}_{\{\varepsilon\leq |x|\leq1/\varepsilon\}}(x) \Big\|_{\textup{L}_x^p(\mathbb{R}^n)} \lesssim_{n,p,Y} 1
\end{equation}
and
\begin{equation}\label{eq:lowercomp2}
\Big\| \widehat{g}_{n,p,\varepsilon,Y}(\xi) - \mathbbm{i}^{-j} \gamma_{n,j,n/q} Y\Big(\frac{\xi}{|\xi|}\Big) |\xi|^{-n/q} \mathbbm{1}_{\{\varepsilon\leq |\xi|\leq1/\varepsilon\}}(\xi) \Big\|_{\textup{L}_{\xi}^q(\mathbb{R}^n)} \lesssim_{n,p,Y} 1 .
\end{equation}
Consequently, also
\begin{equation}\label{eq:lowercomp3}
\|g_{n,p,\varepsilon,Y}\|_{\textup{L}^p(\mathbb{R}^n)} = \|Y\|_{\textup{L}^p(\mathbb{S}^{n-1})} \Big(2\log\frac{1}{\varepsilon}\Big)^{1/p} + O^{\varepsilon\to0+}_{n,p,Y}(1)
\end{equation}
and
\begin{equation}\label{eq:lowercomp4}
\big\|\widehat{g}_{n,p,\varepsilon,Y}\big\|_{\textup{L}^q(\mathbb{R}^n)} = \gamma_{n,j,n/q} \|Y\|_{\textup{L}^q(\mathbb{S}^{n-1})} \Big(2\log\frac{1}{\varepsilon}\Big)^{1/q} + O^{\varepsilon\to0+}_{n,p,Y}(1) .
\end{equation}
\end{lemma}

We emphasize that the implicit constants in \eqref{eq:lowercomp1}--\eqref{eq:lowercomp4} do not depend on $\varepsilon$.

\begin{proof}
Note that $Y$ extends from $\mathbb{S}^{n-1}$ to the unique solid spherical harmonic $P$ of degree $j$ on $\mathbb{R}^n$ via $P(x)=|x|^j Y(x/|x|)$.
We will construct $g=g_{n,p,\varepsilon,Y}$ as a superposition of dilated Gaussian functions, very similarly as these were employed in \cite[Sections~IV.3--IV.4]{SW71Book}.
Define
\[ g(x) := \frac{2\pi^{j/2+n/2p}}{\Gamma(j/2+n/2p)}\, P(x) \int_{\varepsilon}^{1/\varepsilon} e^{-\pi t^{-2}|x|^2} t^{-n/p-j-1} \,\textup{d}t ; \quad x\in\mathbb{R}^n. \]
This is clearly a Schwartz function.
Dilating formula \eqref{eq:FT_on_Gaussians} we get
\[ f(x) = t^{-n-j} P(x) e^{-\pi t^{-2}|x|^2} \ \implies\ \widehat{f}(\xi) = \mathbbm{i}^{-j} t^j P(\xi) e^{-\pi t^2|\xi|^2} , \]
which enables us to take the Fourier transform of $g$:
\[ \widehat{g}(\xi) = \frac{2\mathbbm{i}^{-j}\pi^{j/2+n/2p}}{\Gamma(j/2+n/2p)}\, P(\xi) \int_{\varepsilon}^{1/\varepsilon} e^{-\pi t^2|\xi|^2} t^{n/q+j-1} \,\textup{d}t ; \quad \xi\in\mathbb{R}^n. \]
An easy computation changing the variables of integration leads to
\begin{align*}
g(x) & = Y\Big(\frac{x}{|x|}\Big) |x|^{-n/p} \frac{1}{\Gamma(j/2+n/2p)} \int_{\pi\varepsilon^{2}|x|^2}^{\pi\varepsilon^{-2}|x|^2} u^{j/2+n/2p-1} e^{-u} \,\textup{d}u, \\
\widehat{g}(\xi) & = \frac{\mathbbm{i}^{-j} \pi^{n/2p-n/2q}}{\Gamma(j/2+n/2p)}\, Y\Big(\frac{\xi}{|\xi|}\Big) |\xi|^{-n/q} \int_{\pi\varepsilon^{2}|\xi|^2}^{\pi\varepsilon^{-2}|\xi|^2} u^{j/2+n/2q-1} e^{-u} \,\textup{d}u \\
& = \mathbbm{i}^{-j} \gamma_{n,j,n/q}\, Y\Big(\frac{\xi}{|\xi|}\Big) |\xi|^{-n/q} \frac{1}{\Gamma(j/2+n/2q)} \int_{\pi\varepsilon^{2}|\xi|^2}^{\pi\varepsilon^{-2}|\xi|^2} u^{j/2+n/2q-1} e^{-u} \,\textup{d}u .
\end{align*}
Since $Y$ is clearly bounded on the unit sphere $\mathbb{S}^{n-1}$, estimate \eqref{eq:lowercomp1} is now reduced to
\[ \Big\| \Big( \int_{\pi\varepsilon^{2}|x|^2}^{\pi\varepsilon^{-2}|x|^2} u^{\beta-1} e^{-u} \,\textup{d}u - \Gamma(\beta) \mathbbm{1}_{\{\varepsilon\leq |x|\leq1/\varepsilon\}}(x) \Big) |x|^{-n/p}  \Big\|_{\textup{L}_x^p(\mathbb{R}^n)} \lesssim_{n,p,\beta} 1 \]
for some $\beta\in(0,\infty)$, while \eqref{eq:lowercomp2} then also follows by interchanging $p$ and $q$.
Moreover, by passing to $n$-dimensional spherical coordinates and using the definition of the gamma function, we see that it remains to establish the following four elementary estimates:
{\allowdisplaybreaks
\begin{subequations}
\begin{align}
& \int_{0}^{\varepsilon} \Big( \int_{\pi\varepsilon^2 r^2}^{\pi\varepsilon^{-2} r^2} u^{\beta-1} e^{-u} \,\textup{d}u \Big)^p \,\frac{\textup{d}r}{r} \lesssim_{p,\beta}1 , \label{eq:lowerlmaux1} \\
& \int_{1/\varepsilon}^{\infty} \Big( \int_{\pi\varepsilon^2 r^2}^{\pi\varepsilon^{-2} r^2} u^{\beta-1} e^{-u} \,\textup{d}u \Big)^p \,\frac{\textup{d}r}{r} \lesssim_{p,\beta}1 , \label{eq:lowerlmaux2} \\
& \int_{\varepsilon}^{1/\varepsilon} \Big( \int_{0}^{\pi\varepsilon^2 r^2} u^{\beta-1} e^{-u} \,\textup{d}u \Big)^p \,\frac{\textup{d}r}{r} \lesssim_{p,\beta}1 . \label{eq:lowerlmaux3} \\
& \int_{\varepsilon}^{1/\varepsilon} \Big( \int_{\pi\varepsilon^{-2} r^2}^{\infty} u^{\beta-1} e^{-u} \,\textup{d}u \Big)^p \,\frac{\textup{d}r}{r} \lesssim_{p,\beta}1 . \label{eq:lowerlmaux4}
\end{align}
\end{subequations}
}

We now provide detailed proofs of \eqref{eq:lowerlmaux1}--\eqref{eq:lowerlmaux4}, even though similar computations appeared in \cite[Section~6]{CDK21}.

\emph{Proof of \eqref{eq:lowerlmaux1}.}
For a fixed $r\in(0,\varepsilon]$ we estimate the inner integral as
\[ \int_{\pi\varepsilon^2 r^2}^{\pi\varepsilon^{-2} r^2} u^{\beta-1} e^{-u} \,\textup{d}u
\lesssim \int_{0}^{\pi\varepsilon^{-2} r^2} u^{\beta-1}\,\textup{d}u \lesssim_\beta (\varepsilon^{-1} r)^{2\beta} , \]
then we raise it to the $p$-th power and integrate in $r$, getting
\[ \varepsilon^{-2\beta p} \int_{0}^{\varepsilon} r^{2\beta p-1} \,\textup{d}r
= \frac{1}{2\beta p} < \infty. \]

\emph{Proof of \eqref{eq:lowerlmaux2}.}
Using integration by parts as many times as needed (depending on $\beta$), we easily obtain the following estimate for the incomplete gamma function:
\begin{equation}\label{eq:incompletegamma}
\int_{x}^{\infty} u^{\beta-1} e^{-u} \,\textup{d}u \lesssim_\beta x^{\beta-1} e^{-x}; \quad x\in[1,\infty);
\end{equation}
also see \cite[Eq.~8.11.1--8.11.3]{NIST} or \cite[Eq.~6.5.32]{ASbook64}. By taking $x=\pi\varepsilon^2 r^2$ for any $r\in[1/\varepsilon,\infty)$, we get
\[ \int_{\pi\varepsilon^2 r^2}^{\pi\varepsilon^{-2} r^2} u^{\beta-1} e^{-u} \,\textup{d}u \lesssim_\beta (\varepsilon r)^{2\beta-2} e^{-\pi\varepsilon^2 r^2}. \]
This is raised to the $p$-th power and integrated in $r$, substituting $s=\pi p \varepsilon^2 r^2$:
\[ \int_{1/\varepsilon}^{\infty} (\varepsilon r)^{(2\beta-2)p} e^{-\pi p \varepsilon^2 r^2} \,\frac{\textup{d}r}{r} 
= \frac{1}{2} (\pi p)^{(1-\beta)p} \int_{\pi p}^{\infty} s^{(\beta-1)p-1} e^{-s} \,\textup{d}s < \infty. \]

\emph{Proof of \eqref{eq:lowerlmaux3}.}
From
\[ \int_{0}^{\pi\varepsilon^2 r^2} u^{\beta-1} e^{-u} \,\textup{d}u
\lesssim \int_{0}^{\pi\varepsilon^2 r^2} u^{\beta-1}\,\textup{d}u \lesssim_\beta (\varepsilon r)^{2\beta} \]
we see that the left hand side of \eqref{eq:lowerlmaux3} is at most a multiple of
\[ \varepsilon^{2\beta p} \int_{0}^{1/\varepsilon} r^{2\beta p-1} \,\textup{d}r
= \frac{1}{2\beta p} < \infty. \]

\emph{Proof of \eqref{eq:lowerlmaux4}.}
Using \eqref{eq:incompletegamma} again we can write
\[ \int_{\pi\varepsilon^{-2} r^2}^{\infty} u^{\beta-1} e^{-u} \,\textup{d}u \lesssim_\beta (\varepsilon^{-1} r)^{2\beta-2} e^{-\pi\varepsilon^{-2} r^2}. \]
Thus, the left hand side of \eqref{eq:lowerlmaux4} is at most a constant times
\[ \int_{\varepsilon}^{\infty} (\varepsilon^{-1} r)^{(2\beta-2)p} e^{-\pi p \varepsilon^{-2} r^2} \,\frac{\textup{d}r}{r} 
= \frac{1}{2} (\pi p)^{(1-\beta)p} \int_{\pi p}^{\infty} s^{(\beta-1)p-1} e^{-s} \,\textup{d}s < \infty. \]
This also completes the proofs of \eqref{eq:lowercomp1} and \eqref{eq:lowercomp2}.

In order to establish \eqref{eq:lowercomp3} we only need to combine \eqref{eq:lowercomp1} with
\[ \Big\| Y\Big(\frac{x}{|x|}\Big) |x|^{-n/p} \mathbbm{1}_{\{\varepsilon\leq |x|\leq1/\varepsilon\}}(x) \Big\|_{\textup{L}_x^p(\mathbb{R}^n)} = \|Y\|_{\textup{L}^p(\mathbb{S}^{n-1})} \Big(2\log\frac{1}{\varepsilon}\Big)^{1/p}. \]
In the same way we verify \eqref{eq:lowercomp4}.
\end{proof}

Observe that the error terms in \eqref{eq:lowercomp1} and \eqref{eq:lowercomp2} are of smaller order in $\varepsilon$ than both of the main terms
\begin{equation}\label{eq:lowercomp5}
Y\Big(\frac{x}{|x|}\Big) |x|^{-n/p} \mathbbm{1}_{\{\varepsilon\leq |x|\leq1/\varepsilon\}}(x)
\end{equation}
and
\begin{equation}\label{eq:lowercomp6}
\mathbbm{i}^{-j} \gamma_{n,j,n/q} Y\Big(\frac{\xi}{|\xi|}\Big) |\xi|^{-n/q} \mathbbm{1}_{\{\varepsilon\leq |\xi|\leq1/\varepsilon\}}(\xi) ,
\end{equation}
norms of which were calculated in \eqref{eq:lowercomp3} and \eqref{eq:lowercomp4}, respectively.
Thus, Lemma~\ref{lm:lowercomp} enables us to think of \eqref{eq:lowercomp6} as an approximation of the Fourier transform of \eqref{eq:lowercomp5} up to a small relative error.
By letting $\varepsilon\to0+$ we would formally recover Bochner's distributional identity \eqref{eq:FT_on_Riesz}, but it is crucial for us to stay within the realm of function spaces $\textup{L}^p(\mathbb{R}^n)$ and $\textup{L}^q(\mathbb{R}^n)$.

\smallskip
We are finally in a position to prove Theorem~\ref{thm:generallower}.

\begin{proof}[Proof of Theorem~\ref{thm:generallower}]
Let us first assume that $p>1$, $q<\infty$.
Take a number $\varepsilon\in(0,1/2]$ and a positive integer $J$. By several applications of Lemma~\ref{lm:lowercomp} we can find a Schwartz function $g$ (depending on  $n,p,\varepsilon,m,J$) such that $g$ differs from
\[ x \mapsto \bigg(\sum_{j=0}^{J} Y_j\Big(\frac{x}{|x|}\Big)\bigg) |x|^{-n/q} \mathbbm{1}_{\{\varepsilon\leq|x|\leq1/\varepsilon\}}(x) \]
in the $\textup{L}^q$ norm by $O_{n,p,m,J}(1)$ as $\varepsilon\to0+$, while $\widehat{g}$ differs from 
\[ \xi \mapsto \bigg(\sum_{j=0}^{J} \mathbbm{i}^{-j} \gamma_{n,j,n/p} Y_j\Big(\frac{\xi}{|\xi|}\Big)\bigg) |\xi|^{-n/p} \mathbbm{1}_{\{\varepsilon\leq|\xi|\leq1/\varepsilon\}}(\xi) \]
in the $\textup{L}^p$ norm by $O_{n,p,m,J}(1)$ as $\varepsilon\to0+$.
(Note that the roles of $p$ and $q$ were interchanged here.)
From the same lemma we also obtain a Schwartz function $f$ (depending on  $n,p,\varepsilon$) such that $f$ differs from
\[ x \mapsto |x|^{-n/p} \mathbbm{1}_{\{\varepsilon\leq|x|\leq1/\varepsilon\}}(x) \]
in the $\textup{L}^p$ norm by $O_{n,p}(1)$ as $\varepsilon\to0+$, while $\widehat{f}$ differs from 
\[ \xi \mapsto \gamma_{n,0,n/q} |\xi|^{-n/q} \mathbbm{1}_{\{\varepsilon\leq|\xi|\leq1/\varepsilon\}}(\xi) \]
in the $\textup{L}^q$ norm by $O_{n,p}(1)$ as $\varepsilon\to0+$.
Using Plancherel's theorem we bound
\begin{align*}
& \|T_{m}\|_{\textup{L}^p(\mathbb{R}^n)\to\textup{L}^p(\mathbb{R}^n)} 
\geq \frac{|\langle T_{m} f, g\rangle_{\textup{L}^2(\mathbb{R}^n)}|}{\|f\|_{\textup{L}^p(\mathbb{R}^n)} \|g\|_{\textup{L}^q(\mathbb{R}^n)}}
= \frac{|\langle m \widehat{f}, \widehat{g}\rangle_{\textup{L}^2(\mathbb{R}^n)}|}{\|f\|_{\textup{L}^p(\mathbb{R}^n)} \|g\|_{\textup{L}^q(\mathbb{R}^n)}} \\ 
& = \frac{|\langle \gamma_{n,0,n/q} m, \sum_{j=0}^{J} \mathbbm{i}^{-j} \gamma_{n,j,n/p} Y_j \rangle_{\textup{L}^2(\mathbb{S}^{n-1})}| \,2\log(1/\varepsilon) + o_{n,p,m,J}^{\varepsilon\to0+}(\log(1/\varepsilon))}{ \sigma(\mathbb{S}^{n-1})^{1/p} \|\sum_{j=0}^{J} Y_j\|_{\textup{L}^q(\mathbb{S}^{n-1})} \,2\log(1/\varepsilon) + o_{n,p,m,J}^{\varepsilon\to0+}(\log(1/\varepsilon))} \\
& = \frac{\gamma_{n,0,n/q} |\langle m, \sum_{j=0}^{J} \mathbbm{i}^{-j} \gamma_{n,j,n/p} Y_j \rangle_{\textup{L}^2(\mathbb{S}^{n-1})}| + o_{n,p,m,J}^{\varepsilon\to0+}(1)}{ \sigma(\mathbb{S}^{n-1})^{1/p} \|\sum_{j=0}^{J} Y_j\|_{\textup{L}^q(\mathbb{S}^{n-1})} + o_{n,p,m,J}^{\varepsilon\to0+}(1)} .
\end{align*}
We first take the limit as $\varepsilon\to0+$ to obtain
\begin{equation}\label{eq:propauxlim}
\|T_{m}\|_{\textup{L}^p(\mathbb{R}^n)\to\textup{L}^p(\mathbb{R}^n)} 
\geq \frac{\gamma_{n,0,n/q}}{\sigma(\mathbb{S}^{n-1})^{1/p}}\, \frac{|\langle m, \sum_{j=0}^{J} \mathbbm{i}^{-j} \gamma_{n,j,n/p} Y_j \rangle_{\textup{L}^2(\mathbb{S}^{n-1})}|}{ \|\sum_{j=0}^{J} Y_j\|_{\textup{L}^q(\mathbb{S}^{n-1})} } 
\end{equation}
and then let $J\to\infty$ using conditions \eqref{it:propb} and \eqref{it:propc}.
This proves \eqref{eq:propmainlowsh} and, in combination with \eqref{eq:gammaasympt2}, also \eqref{eq:propmainlow}.

In order to prove \eqref{eq:propweaklow}, assume hypotheses \eqref{it:propa}--\eqref{it:propc} of the theorem with $p=1$, $q=\infty$.
For every pair of conjugated exponents $p\in(1,2]$, $q\in[2,\infty)$ we repeat the prevoius part of the proof leading to \eqref{eq:propauxlim}.
Then we borrow a trick from \cite[Subsection~4.2]{Dra21} (also see \cite[pp.~496--497]{DPV06}) and use a sharp form of the Marcinkiewicz interpolation theorem (see \cite[Theorem~1.3.2]{Gra14book}) together with a trivial estimate on $\textup{L}^2(\mathbb{R}^n)$ to get
\[ \|T_m\|_{\textup{L}^p(\mathbb{R}^n)\to\textup{L}^p(\mathbb{R}^n)} \lesssim \frac{1}{(p-1)^{1/p}} \|T_m\|_{\textup{L}^1(\mathbb{R}^n)\to\textup{L}^{1,\infty}(\mathbb{R}^n)}^{2/p-1} \]
for $1<p<3/2$. Combining this with \eqref{eq:propauxlim}, we conclude
\begin{align}
& \|T\|_{\textup{L}^1(\mathbb{R}^n)\to\textup{L}^{1,\infty}(\mathbb{R}^n)} \nonumber \\
& \gtrsim \lim_{p\to1+} \bigg( (p-1)^{1/p} \,\frac{\gamma_{n,0,n/q}}{\sigma(\mathbb{S}^{n-1})^{1/p}}\, \frac{|\langle m, \sum_{j=0}^{J} \mathbbm{i}^{-j} \gamma_{n,j,n/p} Y_j \rangle_{\textup{L}^2(\mathbb{S}^{n-1})}|}{ \|\sum_{j=0}^{J} Y_j\|_{\textup{L}^q(\mathbb{S}^{n-1})} } \bigg)^{p/(2-p)} \nonumber \\
& = \frac{1}{n}\, \frac{|\langle m, \sum_{j=0}^{J} \mathbbm{i}^{-j} \gamma_{n,j,n} Y_j \rangle_{\textup{L}^2(\mathbb{S}^{n-1})}|}{ \|\sum_{j=0}^{J} Y_j\|_{\textup{L}^{\infty}(\mathbb{S}^{n-1})} }, \label{eq:propauxlim2}
\end{align}
where we used
\begin{align*}
& \lim_{p\to1+} (p-1)^{1/p} \gamma_{n,0,n/q}
= \lim_{p\to1+} (p-1)^{1/p} \pi^{n/2-n/q} \frac{\Gamma(n/2q)}{\Gamma(n/2p)} \\
& = \Big(\lim_{p\to1+} (p-1)^{1/p} q\Big) \Big(\lim_{p\to1+} \frac{2\pi^{n/2-n/q} \Gamma(n/2q+1)}{n\Gamma(n/2p)} \Big)
= \frac{2\pi^{n/2}}{n\Gamma(n/2)}
\end{align*}
and $\sigma(\mathbb{S}^{n-1}) = 2\pi^{n/2}/\Gamma(n/2)$.
Finally, we take limits in \eqref{eq:propauxlim2} as $J\to\infty$ using conditions \eqref{it:propb} and \eqref{it:propc} as before.
\end{proof}

The condition from part \eqref{it:propb} of Theorem~\ref{thm:generallower} is not always easy to verify, since convergence in $\textup{L}^q$ for $q>2$ is typically trickier than $\textup{L}^2$ convergence. The following remark can often be of some help in that matter.

\begin{remark}\label{rem:on_convergence}
If the sequence $(Y_j)_{j=0}^{\infty}$ from Theorem~\ref{thm:generallower} satisfies assumption \eqref{it:propa} and a stronger condition
\begin{equation}\label{eq:condstronger}
\sum_{j=1}^{\infty} j^n \|Y_j\|_{\textup{L}^2(\mathbb{S}^{n-1})}^2 < \infty,
\end{equation}
then convergence claims from \eqref{it:propb} and \eqref{it:propc} are automatically satisfied for every pair of conjugated exponents $p\in[1,2]$ and $q\in[2,\infty]$.
Indeed, recalling \eqref{eq:defcoeffgamma} and \eqref{eq:gammaasympt} we observe
\[ \gamma_{n,j,n/p}^2 \|Y_j\|_{\textup{L}^2(\mathbb{S}^{n-1})}^2 
\lesssim_{n} j^n \|Y_j\|_{\textup{L}^2(\mathbb{S}^{n-1})}^2. \]
Moreover, we apply the endpoint case of Sogge's inequality \eqref{eq:Sogge_est} and the Cauchy--Schwarz inequality to get
\[ \sum_{j=1}^{\infty} \| Y_j \|_{\textup{L}^{\infty}(\mathbb{S}^{n-1})} 
\lesssim_{n} \sum_{j=1}^{\infty} j^{n/2-1} \| Y_j \|_{\textup{L}^{2}(\mathbb{S}^{n-1})} 
\lesssim_{n} \Big(\sum_{j=1}^{\infty} j^n \| Y_j \|_{\textup{L}^2(\mathbb{S}^{n-1})}^2 \Big)^{1/2} . \]
Then we use \eqref{eq:condstronger} and remember that convergence in $\textup{L}^\infty(\mathbb{S}^{n-1})$ implies convergence in every $\textup{L}^q(\mathbb{S}^{n-1})$.

Another interesting observation, which we do not need in the later text, is that, for $1\leq p<2(n+2)/(n+4)$, assumptions \eqref{it:propa} and \eqref{it:propc} imply assumption \eqref{it:propb}. This is easily seen just as before, only applying a larger range of Sogge's estimates \eqref{eq:Sogge_est}.
\end{remark}

Even though Remark~\ref{rem:on_convergence} enables easy verification of the conditions of Theorem~\ref{thm:generallower}, the above reasoning does not necessarily give sharp upper bounds on the quantity $\|u\|_{\textup{L}^{q}(\mathbb{S}^{n-1})}$. Controlling this number sometimes requires significant extra work; see Lemma~\ref{lm:computeu} below.


\section{Proof of the lower bounds in Theorem~\ref{thm:main}}
\label{sec:thm2lower}
In this section we apply Theorem~\ref{thm:generallower} to establish the part~\eqref{part:thmpartb} of Theorem~\ref{thm:main}.
Our symbol $m_{\phi}^{k}$ will be a ``smoothed'' version of
\begin{equation}\label{eq:nonsmoothsymb}
(\xi_1,\ldots,\xi_n) \mapsto \prod_{i=1}^{n/2} \Big(\frac{\xi_{2i-1}+\mathbbm{i}\xi_{2i}}{|\xi_{2i-1}+\mathbbm{i}\xi_{2i}|}\Big)^{k} .
\end{equation}

On the one hand, this smoothing is required for two reasons.
First, $\textup{C}^\infty$ smoothness was imposed in the formulation of Problem~\ref{prob:Mazya}, which we we address. Second, for $n\geq4$ the non-smooth symbol appearing in \eqref{eq:nonsmoothsymb} is singular on the union of two-dimensional coordinate planes, so the part~\eqref{part:thmparta} of Theorem~\ref{thm:main} does not apply and we do not even have clear upper bounds for the associated multiplier operators.

On the other hand, smoothing of the symbol significantly complicates analysis of lower bounds by destroying the tensor product structure. These complications are detailed in Remark~\ref{rem:smoothing}, which rules out the possibility of testing $T_{\phi}^{k}$ on examples of functions that are elementary tensor products with respect to $\mathbb{R}^2\times\cdots\times\mathbb{R}^2$.

In accordance with Remark~\ref{rem:duality}, in the remaining text we always assume
\[ \lambda\geq1, \quad p\in[1,2], \quad q\in[2,\infty], \]
and that $p$ and $q$ are conjugate exponents.

\subsection{Two dimensions}
This short subsection is logically redundant, both because the next subsection covers all even-dimensional spaces $\mathbb{R}^n$, and because Theorem~\ref{thm:cos} provides yet another example of a phase that leads to the ``worst possible'' asymptotics. 
We include it for reader's convenience: to illustrate the main idea of proof with absence of many subtle technical complications arising in dimensions $n\geq4$.

Take $\delta\in(0,\pi)$; we will choose it a bit later. Define $\phi$ in polar coordinates as
\[ \phi(e^{\mathbbm{i}\varphi}) := \varphi; \quad \varphi\in(-\pi+\delta,\pi-\delta) . \]
This still leaves $\phi\colon\mathbb{S}^1\to\mathbb{R}$ undefined on a circular arc of length $2\delta$, but we can define it arbitrarily there, only taking care that $\phi$ is $\textup{C}^{\infty}$ on the whole circle $\mathbb{S}^1$.
For a positive integer $k$ we also define
\[ v^{(k)}(e^{\mathbbm{i}\varphi}) := e^{\mathbbm{i}k\varphi}; \quad \varphi\in\mathbb{R}. \]
This is clearly a spherical harmonic on $\mathbb{S}^1$ of degree $k$, as it is obtained by restricting the harmonic polynomial $(\xi_1,\xi_2) \mapsto (\xi_1+\mathbbm{i}\xi_2)^k$ to the unit circle. With the intention of applying Theorem~\ref{thm:generallower}, we also set
\[ u^{(k)} := \mathbbm{i}^k \gamma_{2,k,2/p}^{-1} v^{(k)}, \]
so that functions $u=u^{(k)}$ and $v=v^{(k)}$ trivially satisfy assumptions \eqref{it:propa}--\eqref{it:propc}.
Note that for every $\varphi\in(-\pi+\delta,\pi-\delta)$ we have
\[ m_{\phi}^{k}(e^{\mathbbm{i}\varphi}) = e^{\mathbbm{i}k\phi(e^{\mathbbm{i}\varphi})} = e^{\mathbbm{i}k\varphi} = v^{(k)}(e^{\mathbbm{i}\varphi}), \]
so, in particular,
\[ \|m_{\phi}^{k}-v^{(k)}\|_{\textup{L}^2(\mathbb{S}^1)} \leq 2\sqrt{\delta}. \]
For any $p\in[1,2]$ we now have
\begin{align*}
\frac{|\langle m_{\phi}^{k}, v^{(k)} \rangle_{\textup{L}^2(\mathbb{S}^{1})}|}{\|u^{(k)}\|_{\textup{L}^{q}(\mathbb{S}^{1})}}
& \geq \frac{\|v^{(k)}\|_{\textup{L}^2(\mathbb{S}^{1})}^2 - \|m_{\phi}^{k} - v^{(k)}\|_{\textup{L}^2(\mathbb{S}^{1})} \|v^{(k)}\|_{\textup{L}^2(\mathbb{S}^{1})}}{\gamma_{2,k,2/p}^{-1} \|v^{(k)}\|_{\textup{L}^{q}(\mathbb{S}^{1})}} \\
& \geq (\sqrt{2\pi} - 2\sqrt{\delta}) (2\pi)^{1/2-1/q} \gamma_{2,k,2/p}.
\end{align*}
Finally, we take $\delta=\pi/8$ and recall $\gamma_{2,k,2/p}\sim k^{2/p-1}$, by \eqref{eq:gammaasympt}. 
Theorem~\ref{thm:generallower} applies, so \eqref{eq:propmainlow} and \eqref{eq:propweaklow} respectively give
\[ \|T_{\phi}^{k}\|_{\textup{L}^p(\mathbb{R}^2)\to\textup{L}^p(\mathbb{R}^2)} \gtrsim (q-1) k^{2/p-1} = (q-1) \,k^{2(1/p-1/2)} \]
for $p\in(1,2]$ and
\[ \|T_{\phi}^{k}\|_{\textup{L}^1(\mathbb{R}^2)\to\textup{L}^{1,\infty}(\mathbb{R}^2)} \gtrsim k . \]
These are precisely the desired two-dimensional cases of \eqref{eq:mainlower} and \eqref{eq:weaklower}.

\subsection{Higher dimensions}
\label{subsec:even}
Suppose that $n=2r$ for a positive integer $r$.
We will consider a slightly non-standard coordinatization of $\mathbb{S}^{2r-1}\subset\mathbb{R}^{2r}$.
Denote
\[ \mathbb{S}^{r-1}_{+} := \{ (\omega_1,\ldots,\omega_r)\in(0,\infty)^r : \omega_1^2+\cdots+\omega_r^2=1 \}, \quad D_{+} := (-\pi,\pi)^r \times \mathbb{S}^{r-1}_{+}. \]
Transformation $\Psi \colon D_{+} \to \mathbb{S}^{2r-1}$ is defined as
\[ \Psi(\varphi_1,\ldots,\varphi_r,\omega) = (\omega_1\cos\varphi_1,\, \omega_1\sin\varphi_1,\, \ldots,\, \omega_r\cos\varphi_r,\, \omega_r\sin\varphi_r), \]
where $\varphi_1,\ldots,\varphi_r\in(-\pi,\pi)$ and $\omega=(\omega_1,\ldots,\omega_r)\in\mathbb{S}^{r-1}_{+}$.
This is a $\textup{C}^\infty$ diffeomorphism onto its image, and the complement of its image is a negligible subset of $\mathbb{S}^{2r-1}$ with respect to the surface measure $\sigma_{2r-1}$.
Moreover, this parametrization $\Psi$ enables us to write the infinitesimal element of the surface measure on $\mathbb{S}^{2r-1}$ as
\begin{equation}\label{eq:PsiJacobian}
\textup{d}\varphi_1 \cdots \textup{d}\varphi_r \,\omega_1 \cdots \omega_r \,\textup{d}\sigma_{r-1}(\omega) .
\end{equation}
In the case $r=1$ one simply needs to disregard any occurrence of $\mathbb{S}^{r-1}_{+}$.

For any $\delta>0$ we also denote
\[ \mathbb{S}^{r-1}_{\delta}:=\mathbb{S}^{r-1}_{+}\cap(\delta,\infty)^r, \quad D_\delta := (-\pi+\delta,\pi-\delta)^r \times \mathbb{S}^{r-1}_{\delta}. \]
Now fix a parameter $0<\delta<1/r$ (to be chosen later) and define $\Phi\colon D_{+}\to\mathbb{R}$ by setting
\[ \Phi(\varphi_1,\ldots,\varphi_r,\omega) :=
\begin{cases}
\varphi_1 + \cdots + \varphi_r & \text{on } D_{\delta}, \\
0 & \text{on } D_{+}\setminus D_{\delta/2},
\end{cases} \]
and choosing its values on $D_{\delta/2}\setminus D_{\delta}$ quite arbitrarily, only taking care that $\Phi$ remains $\textup{C}^{\infty}$.
We can finally define the desired phase function $\phi\colon\mathbb{S}^{2r-1}\to\mathbb{R}$ as
\[ \phi := \begin{cases}
\Phi\circ\Psi^{-1} & \text{on } \Psi(D_{+}), \\
0 & \text{on } \mathbb{S}^{2r-1}\setminus\Psi(D_{+}).
\end{cases} \]
Since the composition $\Phi\circ\Psi^{-1}$ vanishes outside the closure of $\Psi(D_{\delta/2})\subset\mathbb{S}^{2r-1}$, we clearly see that $\phi$ is $\textup{C}^\infty$ on the whole sphere.
We can call
\[ M_\delta := \Big\{\xi\in\mathbb{R}^n\setminus\{\mathbf{0}\} \,:\, \frac{\xi}{|\xi|}\in \Psi(D_\delta)\Big\} \]
the \emph{major cone}, while its complement $\mathbb{R}^n\setminus M_\delta$ is a certain \emph{exceptional set}.
For a positive integer $k$ and any $\xi\in M_\delta$, denoting $\xi/|\xi|=\Psi(\varphi_1,\ldots,\varphi_r,\omega)$ we can write
\[ m_{\phi}^{k}(\xi) = e^{\mathbbm{i} k \phi(\xi/|\xi|)} = e^{\mathbbm{i} k \Phi(\varphi_1,\ldots,\varphi_r,\omega)} 
= \prod_{j=1}^{r} e^{\mathbbm{i} k \varphi_j} , \]
so for every $\xi\in M_\delta$ the symbol simplifies as
\begin{equation}\label{eq:finalsymbol}
m_{\phi}^{k}(\xi) = \prod_{i=1}^{r} \Big(\frac{\xi_{2i-1}+\mathbbm{i}\xi_{2i}}{|\xi_{2i-1}+\mathbbm{i}\xi_{2i}|}\Big)^{k}.
\end{equation}
Once again we remark that it was necessary to smooth out the expression \eqref{eq:finalsymbol} on the exceptional set.
For a generic point $\xi=(\xi_1,\xi_2,\ldots,\xi_{2r-1},\xi_{2r})\in\mathbb{R}^{2r}$ we will also write
\[ \zeta_i := \xi_{2i-1} + \mathbbm{i}\xi_{2i}; \quad j=1,2,\ldots,r \]
and identify pairs $(\xi_{2i-1},\xi_{2i})$ with complex numbers $\zeta_i$.

We choose a particular $\textup{L}^2$ function on the sphere, $\widetilde{m}^{(k)}$, acting on $r$ complex variables and defined as
\[ \widetilde{m}^{(k)}(\zeta_1,\ldots,\zeta_r) := \prod_{i=1}^{r} \Big(\frac{\zeta_i}{|\zeta_i|}\Big)^{k} . \]
The reader can recognize it simply as the right hand side of \eqref{eq:finalsymbol}, i.e., the non-smooth version of the symbol.
Every homogeneous polynomial on $\mathbb{R}^{2r}$ can be written as a homogeneous polynomial in terms of $\zeta_1$, $\bar{\zeta_1}$, \ldots, $\zeta_r$, $\bar{\zeta_r}$.  
By integrating over the sphere we see that $\widetilde{m}^{(k)}$ is orthogonal in $\textup{L}^2(\mathbb{S}^{2r-1})$ to any such monomial that is \underline{not} of the form
\begin{align}
P_{k_1,\ldots,k_r}(\zeta_1,\ldots,\zeta_r) & := \zeta_1^{k+k_1} \bar{\zeta_1}^{k_1} \cdots \zeta_r^{k+k_r} \bar{\zeta_r}^{k_r} \nonumber \\
& = \bigg(\prod_{j=1}^{r} \Big(\frac{\zeta_j}{|\zeta_j|}\Big)^{k}\bigg)\, |\zeta_1|^{k+2k_1}\cdots |\zeta_r|^{k+2k_r} \label{eq:harmform}
\end{align}
for some nonnegative integers $k_1,\ldots,k_r$.
In other words, spherical harmonics from the orthogonal expansion of $\widetilde{m}^{(k)}$ are necessarily linear combinations of \eqref{eq:harmform}.
Also note that $P_{k_1,\ldots,k_r}$ has degree $rk+2k_1+\cdots+2k_r$.

By the previous discussion, the expansion of $\widetilde{m}^{(k)}$ into spherical harmonics is of the form
\[ \widetilde{m}^{(k)} = \sum_{j=rk}^{\infty} \widetilde{Y}_{j}^{(k)}, \]
where each $\widetilde{Y}_{j}^{(k)}$ is a spherical harmonic of degree $j$.
Let us set
\[ Y_{j}^{(k)} := \mathbbm{i}^j \gamma_{n,j,0} \widetilde{Y}_{j}^{(k)} \]
for every integer $j\geq rk$ and define another spherical function by
\begin{equation}\label{eq:seriesdefu}
u^{(k)} := \sum_{j=rk}^{\infty} Y_{j}^{(k)}.
\end{equation}

\begin{lemma}\label{lm:computeu}
Functions $u^{(k)}$ satisfy
\begin{equation}\label{eq:ubound}
\|u^{(k)}\|_{\textup{L}^{1}(\mathbb{S}^{2r-1})} 
\sim_r \|u^{(k)}\|_{\textup{L}^{2}(\mathbb{S}^{2r-1})} 
\sim_r \|u^{(k)}\|_{\textup{L}^{\infty}(\mathbb{S}^{2r-1})} \sim_r k^{-r}
\end{equation}
for every positive integer $k$.
\end{lemma}

\begin{proof}
By property \eqref{eq:FT_general_SW}, the Fourier transform of
\[ K(\xi) = \operatorname{p.v.} \overline{\widetilde{m}^{(k)}\Big(\frac{\xi}{|\xi|}\Big)} \,|\xi|^{-n} \]
is given by
\begin{equation}\label{eq:magicfourier}
\widehat{K}(\xi) = \overline{u^{(k)}\Big(\frac{\xi}{|\xi|}\Big)} .
\end{equation}
Thus, once we compute $\widehat{K}$, we will be able to read off the function $u^{(k)}$ from \eqref{eq:magicfourier}.
The following calculations are much in the spirit of the proofs of Theorems~3.3 and 3.10 from \cite[Chapter~IV]{SW71Book}.

Recalling
\[ K(\zeta_1,\ldots,\zeta_r) = \operatorname{p.v.} \bigg(\prod_{i=1}^{r} \Big(\frac{\zeta_i}{|\zeta_i|}\Big)^{-k}\bigg) \Big(\sum_{i=1}^{r}|\zeta_i|^2\Big)^{-r} \]
and using
\begin{equation*}
\big(|\zeta_1|^2+\cdots+|\zeta_r|^2\big)^{-r} = \frac{2\pi^{r}}{(r-1)!} \int_{0}^{\infty} e^{-\pi t^{-2} (|\zeta_1|^2+\cdots+|\zeta_r|^2)} \,\frac{\textup{d}t}{t^{2r+1}},
\end{equation*}
we can write
\[ K(\zeta_1,\ldots,\zeta_r) = \frac{2\pi^{r}}{(r-1)!} \int_{0}^{\infty} \bigg(\prod_{i=1}^{r} \Big(\frac{\zeta_i}{|\zeta_i|}\Big)^{-k} e^{-\pi t^{-2} |\zeta_i|^2}\bigg) \,\frac{\textup{d}t}{t^{2r+1}}. \]
Let us first compute the Fourier transform of an auxiliary function
\[ f\colon\mathbb{C}\to\mathbb{C}, \quad f(\zeta) = \Big(\frac{\zeta}{|\zeta|}\Big)^{-k} e^{-\pi t^{-2} |\zeta|^2} \]
by changing variables
\[ \zeta = \rho e^{i\varphi}, \ \zeta' = \rho' e^{i\varphi'}; \quad \rho,\rho'\in(0,\infty), \ \varphi,\varphi'\in[0,2\pi) \]
and writing
{\allowdisplaybreaks
\begin{align*}
\widehat{f}(\zeta') & = \int_{\mathbb{C}} f(\zeta) e^{-2\pi \mathbbm{i} \zeta\cdot\zeta'} \,\textup{d}\zeta \\
& = \int_{0}^{\infty} \int_{0}^{2\pi}  e^{-\mathbbm{i}k\varphi} e^{-\pi t^{-2} \rho^2} e^{-2\pi\mathbbm{i}\rho \rho' \cos(\varphi-\varphi')} \rho \,\textup{d}\rho \,\textup{d}\varphi \\
& \quad \big[\text{change variable $\varphi\rightarrow\varphi+\varphi'+\pi/2$}\big] \\
& = \mathbbm{i}^{-k} e^{-\mathbbm{i}k\varphi'} \int_{0}^{\infty} \Big(\int_{0}^{2\pi}e^{-\mathbbm{i}k\varphi+2\pi\mathbbm{i}\rho \rho' \sin\varphi}\,\textup{d}\varphi\Big) e^{-\pi t^{-2} \rho^2} \rho \,\textup{d}\rho \\
& = \mathbbm{i}^{-k} e^{-\mathbbm{i}k\varphi'} 2\pi \int_{0}^{\infty} J_{k}(2\pi\rho\rho') e^{-\pi t^{-2} \rho^2} \rho \,\textup{d}\rho \\
& = \mathbbm{i}^{-k} \Big(\frac{\zeta'}{|\zeta'|}\Big)^{-k} 2\pi t^2 \int_{0}^{\infty} J_{k}(2\pi \rho t|\zeta'|) e^{-\pi \rho^2} \rho \,\textup{d}\rho .
\end{align*}
}
Here $J_{k}$ denotes the Bessel function of the first kind of an integer order $k$ and we used Bessel's integral formula,
\begin{equation}\label{eq:Besselintegral}
J_k(x) = \frac{1}{2\pi}\int_{-\pi}^{\pi} e^{\mathbbm{i}(x\sin\tau-k\tau)} \,\textup{d}\tau
= \frac{\mathbbm{i}^{-k}}{2\pi}\int_{-\pi}^{\pi} e^{\mathbbm{i}(x\cos\tau-k\tau)} \,\textup{d}\tau; \quad x\in(0,\infty);
\end{equation}
see \cite[Eq.~10.9.2]{NIST} or \cite[Eq.~9.1.21]{ASbook64}.
Justifying the interchange of the integral in $t$ and the action the Fourier transform in $\zeta_i$ as in the proof of \cite[Chapter~IV, Theorem~4.5]{SW71Book}, we see that
{\allowdisplaybreaks
\begin{align*}
\widehat{K}(\zeta_1,\ldots,\zeta_r) 
& = \mathbbm{i}^{-rk} \bigg(\prod_{i=1}^{r} \Big(\frac{\zeta_i}{|\zeta_i|}\Big)^{-k}\bigg) \frac{2^{r+1}\pi^{2r}}{(r-1)!} \int_{0}^{\infty} \Big( \prod_{i=1}^{r} \int_{0}^{\infty} J_{k}(2\pi \rho t|\zeta_i|) e^{-\pi \rho^2} \rho \,\textup{d}\rho \Big) \,\frac{\textup{d}t}{t} \\
& \quad \big[\text{change variable $t\rightarrow t/|\xi|$, where $\xi=(\zeta_1,\ldots,\zeta_r)$}\big] \\
& = \mathbbm{i}^{-rk} \bigg(\prod_{i=1}^{r} \Big(\frac{\zeta_i}{|\zeta_i|}\Big)^{-k}\bigg)  \frac{2^{r+1}\pi^{2r}}{(r-1)!} \int_{0}^{\infty} \bigg( \prod_{i=1}^{r} \int_{0}^{\infty} J_{k}\Big(\frac{2\pi \rho t|\zeta_i|}{|\xi|}\Big) e^{-\pi \rho^2} \rho \,\textup{d}\rho \bigg) \,\frac{\textup{d}t}{t},
\end{align*}
}
so comparing with \eqref{eq:magicfourier} we obtain
{\allowdisplaybreaks
\begin{align}
u^{(k)}(\zeta_1,\ldots,\zeta_r)
& = \mathbbm{i}^{rk} \bigg(\prod_{i=1}^{r} \Big(\frac{\zeta_i}{|\zeta_i|}\Big)^{k}\bigg) 
\frac{2^{r+1}\pi^{2r}}{(r-1)!} \int_{0}^{\infty} \Big( \prod_{i=1}^{r} \int_{0}^{\infty} J_{k}(2\pi \rho t|\zeta_i|) e^{-\pi \rho^2} \rho \,\textup{d}\rho \Big) \,\frac{\textup{d}t}{t} \nonumber \\
& \quad \big[\text{change variables $\rho\rightarrow\rho/\sqrt{2\pi k}$, $t\rightarrow kt\sqrt{2/\pi}$}\big] \nonumber \\
& = \frac{2\pi^{r} \mathbbm{i}^{rk}}{(r-1)!}  k^{-r} \bigg(\prod_{i=1}^{r} \Big(\frac{\zeta_i}{|\zeta_i|}\Big)^{k}\bigg) \int_{0}^{\infty} \Big( \prod_{i=1}^{r} \int_{0}^{\infty} J_{k}\big(2 \sqrt{k} \rho t|\zeta_i|\big) e^{-\rho^2/2k} \rho \,\textup{d}\rho \Big) \,\frac{\textup{d}t}{t} \label{eq:utoF}
\end{align}
}
for every $(\zeta_1,\ldots,\zeta_r)\in\mathbb{S}^{2r-1}$.

Define $F_k\colon[0,\infty)\to\mathbb{R}$ by
\begin{equation}\label{eq:intdefF1}
F_k(s) := \int_{0}^{\infty} J_{k}\big(2 \sqrt{k} \rho s\big) e^{-\rho^2/2k} \rho \,\textup{d}\rho .
\end{equation}
A simple change of variable gives
\begin{equation}\label{eq:intdefF2}
F_k(s) = \int_{0}^{\infty} J_{k}(\rho) \frac{\rho}{4ks^2} e^{-\rho^2/8k^2s^2}  \,\textup{d}\rho
\end{equation}
for $s\in(0,\infty)$, while $F_k(0)=0$.
Let us show the following uniform bound in $k$,
\begin{equation}\label{eq:Fclaim1}
0\leq F_k(s) \lesssim \min\{s^2,s^{-2}\} \quad\text{for }s\in(0,\infty),\ k\geq3,
\end{equation}
and the equality
\begin{equation}\label{eq:Fclaim2}
\|F_k\|_{\textup{L}^1(0,\infty)} = \int_{0}^{\infty} F_k(s) \,\textup{d}s = \frac{1}{2}\sqrt{\frac{\pi}{2}} \quad\text{for }k\geq1.
\end{equation}

We first address \eqref{eq:Fclaim1}.
Using \cite[Eq.~10.22.54]{NIST} we can evaluate the integral \eqref{eq:intdefF2} in terms of the confluent hypergeometric function $\mathbf{M}$ as
\[ F_k(s) = 2 \Gamma\Big(\frac{k}{2}+1\Big) (2k^2s^2)^{(k+2)/2} e^{-2k^2s^2} \mathbf{M}\Big(\frac{k}{2},k+1,2k^2s^2\Big) . \]
Then we use the integral representation of $\mathbf{M}$ from \cite[Eq.~13.4.1]{NIST} and simplify to get a convenient formula
\begin{equation}\label{eq:Fmagicf0}
F_k(s) = \frac{2^{k/2}k^{k+1}s^k}{\Gamma(k/2)} \int_0^1 e^{-2k^2s^2\tau} \tau^{k/2} (1-\tau)^{k/2-1} \,\textup{d}\tau.
\end{equation}
Using $1-\tau\leq e^{-\tau}$ for $0\leq \tau\leq 1$ we bound \eqref{eq:Fmagicf0} for $k\geq3$ as
{\allowdisplaybreaks
\begin{align*}
F_k(s) & \leq \frac{2^{k/2}k^{k+1}s^k}{\Gamma(k/2)} \int_{0}^{\infty} e^{-2k^2s^2\tau} \tau^{k/2} (e^{-\tau})^{k/2-1} \,\textup{d}\tau \\
& \quad \big[\text{change variable $\tau\rightarrow 2\tau/(4k^2s^2+k-2)$}\big]  \\
& = \frac{2^{k/2}k^{k+1}s^k}{\Gamma(k/2)} \Big(\frac{2}{4k^2s^2+k-2}\Big)^{k/2+1} \Gamma\Big(\frac{k}{2}+1\Big) \\
& = \frac{1}{4s^2} \Big(1 + \frac{k-2}{4k^2 s^2}\Big)^{-k/2-1} .
\end{align*}
}
Bernoulli's inequality gives
\[ \Big(1 + \frac{k-2}{4k^2 s^2}\Big)^{k/4+1/2} \geq 1 + \frac{k^2-4}{16k^2s^2} \geq 1 + \frac{1}{40s^2}, \]
which further controls $F_k(s)$ as
\[ F_k(s) \leq \frac{1}{4s^2} \Big(1 + \frac{1}{40s^2}\Big)^{-2} \leq \min\Big\{400s^2,\frac{1}{4s^{2}}\Big\}. \]
Next, in order to verify \eqref{eq:Fclaim2}, we use \eqref{eq:Fmagicf0} to write $F_k$ as a superposition of the functions 
\[ (0,\infty)\ni s \mapsto s^k e^{-2k^2 \tau s^2}. \]
Integrals of these functions easily compute to
\[ \frac{\Gamma((k+1)/2)}{2^{(k+3)/2} k^{k+1} \tau^{(k+1)/2}}. \]
It remains to evaluate the integral in $\tau$ of this quantity multiplied with a weight appearing in \eqref{eq:Fmagicf0}, thus obtaining the right hand side of \eqref{eq:Fclaim2}. 
This completes the proofs of the auxiliary claims \eqref{eq:Fclaim1} and \eqref{eq:Fclaim2}.
Moreover, by \eqref{eq:Fclaim1} we have
\[ \int_{(0,\infty)\setminus[1/R,R]} F_k(s) \,\textup{d}s < \frac{1}{10} \]
for a sufficiently large number $R>1$ that depends only on the implicit constant in \eqref{eq:Fclaim1}.
Combining this with \eqref{eq:Fclaim2} we also get
\begin{equation}\label{eq:Fclaim3}
\int_{1/R}^{R} F_k(s) \,\textup{d}s \gtrsim 1 \quad\text{for } k\geq1.
\end{equation}

We have all elements to finalize the proof of the lemma.
Note that \eqref{eq:utoF} and the definition \eqref{eq:intdefF1} give
\begin{align}
\|u^{(k)}\|_{\textup{L}^{\infty}(\mathbb{S}^{2r-1})}
& \lesssim_r k^{-r} \sup_{(\zeta_1,\ldots,\zeta_r)\in\mathbb{S}^{2r-1}} \int_{0}^{\infty} \prod_{i=1}^{r} F_k(t|\zeta_i|) \,\frac{\textup{d}t}{t} \nonumber \\
& = k^{-r} \sup_{(\omega_1,\ldots,\omega_r)\in\mathbb{S}^{r-1}_{+}} \int_{0}^{\infty} \prod_{i=1}^{r} F_k(t\omega_i) \,\frac{\textup{d}t}{t}. \label{eq:Faux0}
\end{align}
Fix an arbitrary point $(\omega_1,\ldots,\omega_r)\in\mathbb{S}^{r-1}_{+}$. Let $\omega_{\max}$ be the largest number among $\omega_1,\ldots,\omega_r$.
Clearly $\omega_{\max}\sim_r 1$. By \eqref{eq:Fclaim1} we have 
\[ \int_{0}^{1} \prod_{i=1}^{r} F_k(t\omega_i) \,\frac{\textup{d}t}{t}
\lesssim_r \int_{0}^{1} F_k(t\omega_{\max}) \,\frac{\textup{d}t}{t} 
\lesssim \int_{0}^{1} t \,\frac{\textup{d}t}{t} \lesssim 1 \]
and
\[ \int_{1}^{\infty} \prod_{i=1}^{r} F_k(t\omega_i) \,\frac{\textup{d}t}{t}
\lesssim_r \int_{1}^{\infty} F_k(t\omega_{\max}) \,\frac{\textup{d}t}{t} 
\lesssim \int_{1}^{\infty} t^{-2} \,\frac{\textup{d}t}{t} \lesssim 1, \]
so that \eqref{eq:Faux0} guarantees
\begin{equation}\label{eq:lmauxcl1}
\|u^{(k)}\|_{\textup{L}^{\infty}(\mathbb{S}^{2r-1})}
\lesssim_r k^{-r}.
\end{equation}

Next, from \eqref{eq:utoF} we also see
{\allowdisplaybreaks
\begin{align*}
\|u^{(k)}\|_{\textup{L}^1(\mathbb{S}^{2r-1})}
& \sim_r k^{-r} \int_{\mathbb{S}^{2r-1}} \int_{0}^{\infty} \prod_{i=1}^{r} F_k(t|\zeta_i|) \,\frac{\textup{d}t}{t} \,\textup{d}\sigma_{2r-1}(\zeta_1,\ldots,\zeta_r) \\
& \quad \big[\text{use parametrization $\Psi$ and recall the surface element \eqref{eq:PsiJacobian}}\big] \\
& \sim_r k^{-r} \int_{\mathbb{S}^{r-1}_{+}} \int_{0}^{\infty} \prod_{i=1}^{r} F_k(t\omega_i) \,\frac{\textup{d}t}{t} \,\omega_1\cdots\omega_r\,\textup{d}\sigma_{r-1}(\omega) \\
& \quad \big[\text{substitute $x_i=t\omega_i$ and denote $x=(x_1,\ldots,x_r)$}\big] \\
& = k^{-r} \int_{(0,\infty)^r} \Big(\prod_{i=1}^{r} F_k(x_i)\Big) \,\frac{x_1\cdots x_r}{|x|^{2r}} \,\textup{d}x \\
& \geq k^{-r} \int_{[1/R,R]^r} \Big(\prod_{i=1}^{r} F_k(x_i)\Big) \,\frac{x_1\cdots x_r}{|x|^{2r}} \,\textup{d}x \\
& \gtrsim_{r,R} k^{-r} \Big(\int_{1/R}^{R} F_k(s) \,\textup{d}s \Big)^r ,
\end{align*}
}
remembering that we chose $R>1$ in the discussion preceding \eqref{eq:Fclaim3}.
Finally, estimate \eqref{eq:Fclaim3} gives
\begin{equation}\label{eq:lmauxcl2}
\|u^{(k)}\|_{\textup{L}^1(\mathbb{S}^{2r-1})}
\gtrsim_r k^{-r}.
\end{equation}
Combining \eqref{eq:lmauxcl1} and \eqref{eq:lmauxcl2} we clearly get \eqref{eq:ubound}.
\end{proof}

Now the proof of part~\eqref{part:thmpartb} of Theorem~\ref{thm:main} can be completed using Theorem~\ref{thm:generallower} applied with $u=u^{(k)}$ introduced before and $v=v^{(k)}$ defined as
\[ v^{(k)} := \sum_{j=rk}^{\infty} \mathbbm{i}^{-j} \gamma_{n,j,n/p} Y_{j}^{(k)}
= \sum_{j=rk}^{\infty} \gamma_{n,j,n/p} \gamma_{n,j,0} \widetilde{Y}_{j}^{(k)}. \]
Recall \eqref{eq:finalsymbol}, i.e., we are taking the symbol $m_{\phi}^{k}$ to be a smoothened version of $\widetilde{m}^{(k)}$ in a way that
\begin{equation}\label{eq:easyest1}
\| \widetilde{m}^{(k)} - m_{\phi}^{k} \|_{\textup{L}^2(\mathbb{S}^{n-1})} 
\leq \sigma_{n-1}\big(\mathbb{S}^{n-1}\setminus\Psi(D_{\delta})\big)^{1/2}
= o^{\delta\to0+}_{n}(1) .
\end{equation}
Trivially,
\begin{equation}\label{eq:easyest2}
\|\widetilde{m}^{(k)}\|_{\textup{L}^2(\mathbb{S}^{n-1})} \sim_{n} 1 .
\end{equation}
On the one hand,
{\allowdisplaybreaks
\begin{align*}
\|v^{(k)}\|_{\textup{L}^2(\mathbb{S}^{2r-1})}^2 
& = \sum_{j=rk}^{\infty} 
\gamma_{2r,j,2r/p}^2 \gamma_{2r,j,0}^2 \big\|\widetilde{Y}_{j}^{(k)}\big\|_{\textup{L}^2(\mathbb{S}^{2r-1})}^2 
\sim_{r} \sum_{j=rk}^{\infty}  j^{4r/p-4r} \big\|\widetilde{Y}_{j}^{(k)}\big\|_{\textup{L}^2(\mathbb{S}^{2r-1})}^2 \\
& \leq (rk)^{4r/p-4r} \sum_{j=rk}^{\infty}  \big\|\widetilde{Y}_{j}^{(k)}\big\|_{\textup{L}^2(\mathbb{S}^{2r-1})}^2
= (rk)^{4r/p-4r} \| \widetilde{m}^{(k)}\|_{\textup{L}^2(\mathbb{S}^{2r-1})}^2 ,
\end{align*}
}
so, in combination with \eqref{eq:easyest2}, 
\begin{equation}\label{eq:easyest3}
\|v^{(k)}\|_{\textup{L}^2(\mathbb{S}^{n-1})} \lesssim_{n} k^{n/p-n} .
\end{equation}
On the other hand,
{\allowdisplaybreaks
\begin{align*}
\langle \widetilde{m}^{(k)}, v^{(k)} \rangle_{\textup{L}^2(\mathbb{S}^{2r-1})}
& = \sum_{j=rk}^{\infty} 
\big\langle \mathbbm{i}^{-j} \gamma_{2r,j,2r} Y_{j}^{(k)}, \mathbbm{i}^{-j} \gamma_{2r,j,2r/p} Y_{j}^{(k)}  \big\rangle_{\textup{L}^2(\mathbb{S}^{2r-1})} \\
& = \sum_{j=rk}^{\infty} 
\gamma_{2r,j,2r} \gamma_{2r,j,2r/p} \|Y_{j}^{(k)}\|_{\textup{L}^2(\mathbb{S}^{2r-1})}^2
\sim_{r} \sum_{j=rk}^{\infty} 
j^{2r/p} \|Y_{j}^{(k)}\|_{\textup{L}^2(\mathbb{S}^{2r-1})}^2 \\
& \geq (rk)^{2r/p} \sum_{j=rk}^{\infty}  \|Y_{j}^{(k)}\|_{\textup{L}^2(\mathbb{S}^{2r-1})}^2
= (rk)^{2r/p} \| u^{(k)}\|_{\textup{L}^2(\mathbb{S}^{2r-1})}^2 .
\end{align*}
}
From this and \eqref{eq:ubound} we get
\begin{equation}\label{eq:easyest4}
\langle \widetilde{m}^{(k)}, v^{(k)} \rangle_{\textup{L}^2(\mathbb{S}^{n-1})} \gtrsim_{n} k^{n/p-n}.
\end{equation}
As a consequence of \eqref{eq:easyest1}, \eqref{eq:easyest3}, and \eqref{eq:easyest4} we can choose $\delta>0$ sufficiently small depending on $n$ only, to achieve
\begin{align*}
\langle m_{\phi}^{k}, v^{(k)} \rangle_{\textup{L}^2(\mathbb{S}^{n-1})} 
& \geq \langle \widetilde{m}^{(k)}, v^{(k)} \rangle_{\textup{L}^2(\mathbb{S}^{n-1})} - \| \widetilde{m}^{(k)} - m_{\phi}^{k} \|_{\textup{L}^2(\mathbb{S}^{n-1})} \|v^{(k)}\|_{\textup{L}^2(\mathbb{S}^{n-1})} \\
& \gtrsim_{n} k^{n/p-n}
\end{align*}
for every positive integer $k$.
Combining this with \eqref{eq:ubound}, we can write
\[ \frac{|\langle m_{\phi}^{k}, v^{(k)} \rangle_{\textup{L}^2(\mathbb{S}^{n-1})}|}{\|u^{(k)}\|_{\textup{L}^q(\mathbb{S}^{n-1})}} 
\gtrsim_{n} \frac{|\langle m_{\phi}^{k}, v^{(k)} \rangle_{\textup{L}^2(\mathbb{S}^{n-1})}|}{\|u^{(k)}\|_{\textup{L}^{\infty}(\mathbb{S}^{n-1})}} 
\gtrsim_{n} \frac{k^{n/p-n}}{k^{-n/2}} = k^{n(1/p-1/2)}. \]
Moreover,
\[ \sum_{j=rk}^{\infty} j^{n} \|Y_{j}^{(k)}\|_{\textup{L}^2(\mathbb{S}^{n-1})}^2 
\sim_n \sum_{j=rk}^{\infty} \big\|\widetilde{Y}_{j}^{(k)}\big\|_{\textup{L}^2(\mathbb{S}^{n-1})}^2 
= \| \widetilde{m}^{(k)}\|_{\textup{L}^2(\mathbb{S}^{n-1})}^2 < \infty, \]
so that Remark~\ref{rem:on_convergence} applies and it guarantees $\textup{L}^q$ convergence of the series \eqref{eq:seriesdefu} on $\mathbb{S}^{n-1}$.
We can now apply Theorem~\ref{thm:generallower}; estimates \eqref{eq:propmainlow} and \eqref{eq:propweaklow} respectively give \eqref{eq:mainlower} and \eqref{eq:weaklower} for every positive integer $k$.


\section{Proof of Theorem~\ref{thm:cos}}
\label{sec:proofofcos}
Take $\lambda\geq1$.
Decomposition of the symbol $m_{\cos}^{\lambda}$ into spherical harmonics is simply its Fourier series expansion,
\begin{equation}\label{eq:cosproof0}
m_{\cos}^{\lambda}(e^{\mathbbm{i}\varphi}) = e^{\mathbbm{i}\lambda\cos\varphi} = J_{0}(\lambda) + 2\sum_{j=1}^{\infty} \mathbbm{i}^j J_{j}(\lambda) \cos j\varphi,
\end{equation}
where $J_j$ are again Bessel functions of the first kind and we used formula \eqref{eq:Besselintegral}.
Taking the real part of the identity \eqref{eq:cosproof0} we get
\begin{equation}\label{eq:cosproof1}
\cos(\lambda\cos\varphi) = J_{0}(\lambda) + 2\sum_{l=1}^{\infty} (-1)^l J_{2l}(\lambda) \cos 2l\varphi
\end{equation}
and then changing $\varphi\rightarrow\varphi-\pi/2$ we also obtain
\begin{equation}\label{eq:cosproof2}
\cos(\lambda\sin\varphi) = J_{0}(\lambda) + 2\sum_{l=1}^{\infty} J_{2l}(\lambda) \cos 2l\varphi.
\end{equation}

The following technical lemma deals with sums of Bessel functions with even-integer orders, and it contains all information we need about the Fourier coefficients of $m_{\cos}^{\lambda}$.

\begin{lemma}
We have
\begin{align}
\sum_{l=1}^{\infty} l^2 J_{2l}(\lambda)^2 \gtrsim \lambda^2, \label{eq:cosproof4} \\
\sum_{l=1}^{\infty} l^4 J_{2l}(\lambda)^2 \lesssim \lambda^4, \label{eq:cosproof7}
\end{align}
and for every $p\in[1,2]$ we have
\begin{equation}\label{eq:cosproof6}
\sum_{l=1}^{\infty} l^{2/p-1} J_{2l}(\lambda)^2 \gtrsim \lambda^{2/p-1}.
\end{equation}
\end{lemma}

\begin{proof}
Since $m_{\cos}^{\lambda}$ is $\textup{C}^\infty$ on $\mathbb{S}^1$, its coefficients $(J_{j}(\lambda))_{j=1}^{\infty}$ decay faster than $O_{\lambda}^{j\to\infty}(j^{-A})$ for every $A>0$. 
We can differentiate the series in \eqref{eq:cosproof1} term-by-term once or twice with respect to $\varphi$, which gives us two more identities:
\begin{align}
\lambda\sin(\lambda\cos\varphi) \sin\varphi & = 4 \sum_{l=1}^{\infty} (-1)^{l-1} l J_{2l}(\lambda) \sin 2l\varphi, \label{eq:cosproof3} \\
-\lambda^2\cos(\lambda\cos\varphi) \sin^2\varphi + \lambda \sin(\lambda\cos\varphi) \cos\varphi
& = 8 \sum_{l=1}^{\infty} (-1)^{l-1} l^2 J_{2l}(\lambda) \cos 2l\varphi. \label{eq:cosproof8}
\end{align}

On the one hand, an application of Parseval's formula to the function on the left hand side of \eqref{eq:cosproof3} yields 
{\allowdisplaybreaks
\begin{align*}
\sum_{l=1}^{\infty} l^2 J_{2l}(\lambda)^2 
& \sim \lambda^2 \int_{-\pi}^{\pi} \sin^2(\lambda\cos\varphi) \sin^2\varphi \,\textup{d}\varphi \\
& \geq 2\lambda^2 \int_{0}^{\pi} \sin^2(\lambda\cos\varphi) \sin^3\varphi \,\textup{d}\varphi \\
& \quad \big[\text{substitute $s=\cos\varphi$}\big] \\
& = 2\lambda^2 \int_{-1}^{1} (1-s^2) \sin^2 \lambda s \,\textup{d}s
= 2\lambda^2 \Big(\frac{2}{3} + \frac{\cos2\lambda}{2\lambda^2} - \frac{\sin2\lambda}{4\lambda^3}\Big),
\end{align*}
}
which confirms \eqref{eq:cosproof4}.
On the other hand, Parseval's formula applied to \eqref{eq:cosproof8} clearly gives \eqref{eq:cosproof7}.
Finally, for any fixed $p\in[1,2]$, by H\"{o}lder's inequality for sums we have
\[ \Big(\sum_{l=1}^{\infty} l^2 J_{2l}(\lambda)^2\Big)^{5/2-1/p} \leq \Big(\sum_{l=1}^{\infty} l^{2/p-1} J_{2l}(\lambda)^2\Big) \Big(\sum_{l=1}^{\infty} l^4 J_{2l}(\lambda)^2\Big)^{3/2-1/p}, \]
which, in combination with \eqref{eq:cosproof4} and \eqref{eq:cosproof7}, gives \eqref{eq:cosproof6}.
\end{proof}

Let us return to the proof of Theorem~\ref{thm:cos}.
We are about to apply Theorem~\ref{thm:generallower} with the function
\[ u^{(\lambda)}(e^{\mathbbm{i}\varphi}) := \cos(\lambda\sin\varphi) - J_{0}(\lambda), \]
recalling that its Fourier expansion can be seen from \eqref{eq:cosproof2}.
The corresponding function $v^{(\lambda)}$ is then given by
\[ v^{(\lambda)}(e^{\mathbbm{i}\varphi}) := 2 \sum_{l=1}^{\infty} (-1)^l \gamma_{2,2l,2/p} J_{2l}(\lambda) \cos 2l\varphi. \]
Clearly, 
\[ \|u^{(\lambda)}\|_{\textup{L}^q(\mathbb{S}^1)} \lesssim \|u^{(\lambda)}\|_{\textup{L}^{\infty}(\mathbb{S}^1)} \leq 2 , \]
while \eqref{eq:gammaasympt} and \eqref{eq:cosproof6} give
\[ \langle m_{\cos}^{\lambda}, v^{(\lambda)}\rangle_{\textup{L}^2(\mathbb{S}^1)} = 4\pi \sum_{l=1}^{\infty} \gamma_{2,2l,2/p} J_{2l}(\lambda)^2 \gtrsim \lambda^{2/p-1}. \]
Since
\[ \frac{|\langle m_{\cos}^{\lambda}, v^{(\lambda)} \rangle_{\textup{L}^2(\mathbb{S}^{1})}|}{\|u^{(\lambda)}\|_{\textup{L}^q(\mathbb{S}^{1})}} 
\gtrsim \lambda^{2/p-1} = \lambda^{2(1/p-1/2)} , \]
estimates \eqref{eq:propmainlow} and \eqref{eq:propweaklow} give the lower bounds in \eqref{eq:cosstrong} and \eqref{eq:cosweak}, respectively, for every positive integer $k$. 

The upper bounds in \eqref{eq:cosstrong} and \eqref{eq:cosweak} are just special cases of \eqref{eq:mainupper} and \eqref{eq:weakupper} in the part~\eqref{part:thmparta} of Theorem~\ref{thm:main}.


\section{Closing remarks}

\begin{remark}\label{rem:smoothing}
Note that the equality \eqref{eq:finalsymbol} only holds on the major cone in $\mathbb{R}^{n}=\mathbb{R}^{2r}$ and changing the symbol on a set of small measure can drastically change the Fourier multiplier norm.
The fact that $m_{\phi}^{k}$ does not exactly split into a tensor product of two-dimensional symbols prevents us from plugging in elementary tensors for ``almost extremizing'' functions. 

Indeed, let us try to test our operator $T_{\phi}^{k}$ on $f=f_{p,r,\varepsilon}$ and $g=g_{k,p,r,\varepsilon}$ given as $r$-fold elementary tensors
\begin{align*}
f(x_1,x_2,\ldots,x_{2r-1},x_{2r}) & := f_2(x_1,x_2) \cdots f_2(x_{2r-1},x_{2r}) , \\
g(x_1,x_2,\ldots,x_{2r-1},x_{2r}) & := g_2(x_1,x_2) \cdots g_2(x_{2r-1},x_{2r}) .
\end{align*}
Here $f_2$ and $g_2$ depend on $k,p,\varepsilon$ and they are chosen as in Lemma~\ref{lm:lowercomp}, i.e., such that 
\begin{align*}
\Big\| f_{2}(x) - |x|^{-2/p} \mathbbm{1}_{\{\varepsilon\leq |x|\leq1/\varepsilon\}}(x) \Big\|_{\textup{L}_x^p(\mathbb{R}^2)} & \lesssim_{p} 1, \\
\Big\| \widehat{f_2}(\xi) - \gamma_{2,0,2/q} |\xi|^{-2/q} \mathbbm{1}_{\{\varepsilon\leq |\xi|\leq1/\varepsilon\}}(\xi) \Big\|_{\textup{L}_{\xi}^q(\mathbb{R}^2)} & \lesssim_{p} 1, \\
\Big\| g_{2}(x) - \Big(\frac{x_1+\mathbbm{i}x_2}{|x_1+\mathbbm{i}x_2|}\Big)^k |x|^{-2/q} \mathbbm{1}_{\{\varepsilon\leq |x|\leq1/\varepsilon\}}(x) \Big\|_{\textup{L}_x^q(\mathbb{R}^2)} & \lesssim_{k,p} 1, \\
\Big\| \widehat{g_2}(\xi) - \mathbbm{i}^{-k} \gamma_{2,k,2/p} \Big(\frac{\xi_1+\mathbbm{i}\xi_2}{|\xi_1+\mathbbm{i}\xi_2|}\Big)^k |\xi|^{-2/p} \mathbbm{1}_{\{\varepsilon\leq |\xi|\leq1/\varepsilon\}}(\xi) \Big\|_{\textup{L}_{\xi}^p(\mathbb{R}^2)} & \lesssim_{k,p} 1 .
\end{align*}
Then we have
\[ \|f\|_{\textup{L}^p(\mathbb{R}^n)} \|g\|_{\textup{L}^q(\mathbb{R}^n)} = \Big(4\pi\log\frac{1}{\varepsilon}\Big)^r + o_{k,p,r}^{\varepsilon\to0+}\Big(\Big(\log\frac{1}{\varepsilon}\Big)^r\Big) \]
and
{\allowdisplaybreaks
\begin{align}
& \langle T_{\phi}^k f, g\rangle_{\textup{L}^2(\mathbb{R}^n)}
= \big\langle m_{\phi}^k \widehat{f}, \widehat{g} \big\rangle_{\textup{L}^2(\mathbb{R}^n)} \nonumber \\
& = \big\langle m_{\phi}^k(\xi) \,\widehat{f}_2(\xi_1,\xi_2) \cdots \widehat{f}_2(\xi_{2r-1},\xi_{2r}), \,\widehat{g}_2(\xi_1,\xi_2) \cdots \widehat{g}_2(\xi_{2r-1},\xi_{2r}) \big\rangle_{\textup{L}^2_{\xi}(\mathbb{R}^n)} \nonumber \\
& = \mathbbm{i}^{rk} \gamma_{2,k,2/p}^r \int_{\mathbb{R}^n} m_{\phi}^k(\xi) \bigg( \prod_{j=1}^r \Big(\frac{\overline{\xi_{2j-1}+\mathbbm{i}\xi_{2j}}}{|\xi_{2j-1}+\mathbbm{i}\xi_{2j}|}\Big)^k |(\xi_{2j-1},\xi_{2j}))|^{-2} \nonumber \\[-2mm]
& \qquad\qquad\qquad\qquad\qquad\qquad\ \  
\mathbbm{1}_{\{\varepsilon\leq|(\xi_{2j-1},\xi_{2j})|\leq1/\varepsilon\}}(\xi_{2j-1},\xi_{2j}) \,\textup{d}\xi_{2j-1}\,\textup{d}\xi_{2j} \bigg) \label{eq:thelowintegral} \\
& \quad + o_{n,k,p}^{\varepsilon\to0+}\Big(\Big(\log\frac{1}{\varepsilon}\Big)^r\Big) \nonumber
\end{align}
}
The integral \eqref{eq:thelowintegral} restricted to the major cone $M_\delta$ equals 
\begin{align*}
& \int_{M_\delta} \prod_{j=1}^r |(\xi_{2j-1},\xi_{2j}))|^{-2} \mathbbm{1}_{\{\varepsilon\leq|(\xi_{2j-1},\xi_{2j})|\leq1/\varepsilon\}}(\xi_{2j-1},\xi_{2j}) \,\textup{d}\xi_{2j-1}\,\textup{d}\xi_{2j} ,
\end{align*}
but, for a fixed $\delta$, this grows only like $\log(1/\varepsilon)$ and not like $(\log(1/\varepsilon))^r$, as it should. The major cone $M_\delta$ is not so much ``major'' in this matter.

Thus, we are better off sticking to functions $f$ and $g$ with more radial symmetry, just as we did before. This is also philosophically in line with the fact that the symbol $m_{\phi}^{k}$ is homogeneous.
\end{remark}

\begin{remark}
The auxiliary function $\widetilde{m}^{(k)}$ used in Subsection~\ref{subsec:even} has a quite complicated expansion into spherical harmonics despite its relatively simple defining formula.
In $n=4$ dimensions this expansion can still be computed explicitly. For simplicity suppose that $k\geq2$ is even. Then
\begin{equation}\label{eq:mtilde4d}
\widetilde{m}^{(k)} = \sum_{\substack{j\geq2k\\ j\text{ divisible by }4}} \widetilde{Y}_j^{(k)} ,
\end{equation}
where
\begin{align*}
\widetilde{Y}_j^{(k)}(\zeta_1,\zeta_2) = & \frac{\binom{j}{j/2}\binom{j/2}{j/4-k/2}(j+1)k}{\binom{j/2}{j/4}\binom{j}{j/2-k}(j/2)(j/2+1)} \\ & \times \zeta_1^k \zeta_2^k \sum_{l=0}^{j/2-k} (-1)^{l} \dbinom{j/2}{j/2-k-l} \dbinom{j/2}{l} |\zeta_1|^{j-2k-2l} |\zeta_2|^{2l}.
\end{align*}
Even though this formula is explicit, it still does not reveal how to compute the associated auxiliary function $u^{(k)}$, which is the work we were doing in the very technical proof of Lemma~\ref{lm:computeu}. However, one can argue that $u^{(k)}$ is ``very close'' to a constant multiple of
\[ \mathbb{C}^2\supset\mathbb{S}^3 \ni (\zeta_1,\zeta_2) \mapsto \Big(\frac{\zeta_1}{|\zeta_1|}\Big)^{k}\Big(\frac{\zeta_2}{|\zeta_2|}\Big)^{k} |\zeta_1|^2 |\zeta_2|^2, \]
and the latter function could have been used for the same purpose, leading to a slightly shorter proof in four dimensions.

In higher dimensions computing the expansion of $\widetilde{m}^{(k)}$ into spherical harmonics seemed impossible to us, or at least not possible in any explicit or practical way.
Also note a pleasant property of \eqref{eq:mtilde4d}: only its every fourth term is nonzero. This property is not retained in higher dimensions, where only every other term in the corresponding expansion is equal to zero. 
Since the passage from $u$ to $v$ in Theorem~\ref{thm:generallower} requires inserting the coefficients $\mathbbm{i}^{-j}\gamma_{n,j,n/p}$, we see that we are, in fact, also inserting $\pm$ signs into the series $\sum_{j=0}^{\infty}Y_j$, which is a very subtle operation.
\end{remark}


\section*{Acknowledgements}
This work was supported in part by the \emph{Croatian Science Foundation} project UIP-2017-05-4129 (MUNHANAP).
We are grateful to Vladimir Maz'ya for his interest in this paper and for suggesting a few references.
We are also grateful to Ton\v{c}i Crmari\'{c} for a reference to the original appearance of Bochner's formula. 


\bibliography{sing_int_asym}{}
\bibliographystyle{plain}

\end{document}